\documentclass[10pt,leqno]{amsart}
\usepackage{bm}
\usepackage{amssymb}
\usepackage{times}
\usepackage{mathrsfs}
\usepackage{enumerate}
\usepackage[dvips]{graphicx,color}

\theoremstyle{plain} 
\newtheorem{theorem}{Theorem}[section] 
 
\newtheorem{proposition}[theorem]{Proposition}
\newtheorem{lemma}[theorem]{Lemma} 

\theoremstyle{definition} 
\newtheorem{definition}[theorem]{Definition}
\newtheorem{remark}[theorem]{Remark}

\DeclareMathOperator{\Div}{div}

\DeclareMathOperator*{\osc}{osc}

\newcommand{\N}{\mathbb{N}}

\newcommand{\R}{\mathbb{R}}

\newcommand{\Sp}{\mathbb{S}}

\numberwithin{equation}{section}
\title{H\"older estimates for solutions of the Cauchy problem for the
porous medium equation with external forces}
\author{Masashi Mizuno}
\address[Masashi Mizuno]{Department of Mathematics, College of Science
and Technology, Nihon University, Tokyo 101-8308 JAPAN}
\email{mizuno@math.cst.nihon-u.ac.jp}

\allowdisplaybreaks[1]
\begin{document}

\begin{abstract}
 We study the interior H\"older regularity problem for weak solutions of
 the porous medium equation with external forces. Since the porous
 medium equation is the typical example of degenerate parabolic
 equations, H\"older regularity is a delicate matter and does not follow
 by classical methods. Caffrelli-Friedman, and
 Caffarelli-Vazquez-Wolansky showed H\"older regularity for the model
 equation without external forces. DiBenedetto and Friedman showed the
 H\"older continuity of weak solutions with some integrability
 conditions of the external forces but they did not obtain the
 quantitative estimates. The quantitative estimates are important for
 studying the perturbation problem of the porous medium equation. We
 obtain the scale invariant H\"older estimates for weak solutions of the
 porous medium equations with the external forces. As a particular case,
 we recover the well known H\"older estimates for the linear heat
 equation.
\end{abstract}

\maketitle


\section{Introduction}
We consider the following degenerate parabolic equation:
\begin{equation}
 \label{eq:1} 
  \left\{
  \begin{aligned}
   \partial_tu-\Delta u^m&=\Div{f}+g,\quad t>0\,,\,x\in\R^n, \\
   u(0,x)&=u_0(x)\geq0,\quad x\in\R^n,
  \end{aligned}
 \right.
\end{equation}
where $m>1$ is a constant,
$u=u(t,x):(0,\infty)\times\R^n\rightarrow\R$ is unknown,
$u_0=u_0(x):\R^n\rightarrow[0,\infty)$,
${f}={f}(t,x):(0,\infty)\times\R^n\rightarrow\R^n$ and
$g=g(t,x):(0,\infty)\times\R^n\rightarrow\R$ are given. For
${f},g\equiv0$, the equation \eqref{eq:1} is called the porous medium
equation. The equation \eqref{eq:1} is a degenerate parabolic equation
since the diffusion coefficient $m u^{m-1}$ may vanish. 
It is well-known that solutions of the degenerate parabolic equation
\eqref{eq:1} are not generally smooth even if the initial datum $u_0$ is
smooth enough. We now introduce the notion of weak solutions.

 \begin{definition}
  For $u_0\in L^1(\R^n)$ and for ${f},g\in
   L^1((0,\infty)\times\R^n)$, we call $u$ a weak solution of
   \eqref{eq:1} if there exists $T>0$ such that
  \begin{enumerate}
   \item $u(t,x)\geq0$\ \ for almost all $(t,x)\in [0,T)\times\R^n$;
   \item $u\in L^\infty(0,T\,;\,L^1(\R^n)\cap L^{m+1}(\R^n))$ with
	 $\nabla u^m\in L^2((0,T)\times\R^n)$;
   \item $u$ satisfies \eqref{eq:1}, namely 
	 for all $\varphi\in C^1(0,T\,;\,C_0^1(\R^n))$ and 
	 for almost all $0<t<T$,
       \begin{multline*}
	\int_{\R^n}u(t)\varphi(t)\,dx
	-\int_0^t\int_{\R^n}u\partial_t\varphi\,d\tau dx
	+\int_0^t\int_{\R^n}
	\nabla u^m\cdot\nabla\varphi\,d\tau dx \\
	=\int_{\R^n}u_0\varphi(0)\,dx
	-\int_0^t\int_{\R^n}f\cdot\nabla\varphi\,d\tau dx
	+\int_0^t\int_{\R^n}g\varphi\,d\tau dx.
       \end{multline*}       
  \end{enumerate}
 \end{definition}
 We remark that the existence of weak solutions is shown by
 Ole{\u\i}nik-Kala{\v{s}}inkov-{\v{C}}{\v{z}}ou~\cite{MR0099834} and
 J.\,L.\,Lions~\cite{MR0259693} (cf. \^Otani~\cite{MR2087494}). Our aim
 in this paper is to obtain a priori H\"older estimates for weak
 solutions of \eqref{eq:1}.

 Caffarelli-Friedman \cite{MR0570687} and Caffarelli-V\'azquez-Wolanski
 \cite{MR0891781} showed H\"older continuity for solutions of the porous
 medium equation. They essentially use a pointwise estimates for the
 derivative of solutions given by Aronson-Benilan~\cite{MR0524760} and
 the comparison principle for the porous medium equation.  For the
 general case with the external force \eqref{eq:1}, the Aronson-Benilan
 type estimate is not known.  In addition, if the equation involves
 non-local effect such as the system with other equations, the
 comparison principle does not generally hold. For instance, we consider
 the following degenerate Keller-Segel system:
 \begin{equation}
  \label{eq:45}
  \left\{
   \begin{aligned}
    \partial_tu-\Delta u^m+\Div(u\nabla\psi)&=0,&\quad&t>0,\ x\in\R^n, \\
    -\Delta\psi+\psi&=u,&\quad&t>0,\ x\in\R^n, \\
    u(0,x)&=u_0(x),&\quad &x\in\R^n.
   \end{aligned}
  \right.
 \end{equation}
 It is known by Sugiyama-Kunii~\cite{MR2235324} that there exists a time
 global bounded weak solution $(u,\psi)$ of \eqref{eq:45} in the case of
 $1\leq m\leq2-\frac2n$ and $n\geq3$ for small initial data. Regularity
 estimates of solutions of \eqref{eq:45} are closely related to the
 large time asymptotic behavior of solutions of \eqref{eq:45}
 (cf. Luckhaus-Sugiyama~\cite{MR2333473}, Ogawa~\cite{MR2482499},
 Ogawa-Mizuno~\cite{O-M}), however comparison principles do not hold for
 \eqref{eq:45}. Therefore it is an worth to derive the regularity of the
 weak solution of \eqref{eq:1} without using the comparison principle.

 On the other hand, DiBenedetto-Friedman~\cite{MR0783531},
 Wiegner~\cite{MR0886719} considered the  $p$-Laplace evolution
 equation:
 \begin{equation}
 \label{eq:2} \left\{
   \begin{aligned}
    \partial_tv-\Div(|\nabla v|^{p-2}\nabla v)&=0,\quad
    t>0\,,\,x\in\R^n, \\
    v(0,x)&=v_0(x),\quad x\in\R^n.
   \end{aligned}
  \right.
 \end{equation}
 The $p$-Laplace evolution equation is a typical example of degenerate
 parabolic equations. They showed the H\"older continuity for the
 gradient of the solutions of \eqref{eq:2} by using the method of
 alternative and intrinsic rescaling. Misawa~\cite{MR1939688} showed the
 gradient H\"older estimates for more general $p$-Laplace evolution
 equations.  We remark that they does not rely on the comparison
 principle for the $p$-Laplace evolution equation
 \eqref{eq:2}. Roughly speaking, the gradient of the solution may be
 regarded to satisfy \eqref{eq:1} with ${f},g\equiv0$ and it seems
 possible to apply their methods for solutions of \eqref{eq:1}. In fact,
 DiBenedetto-Friedman~\cite{MR0783531} showed H\"older continuity for
 solutions of \eqref{eq:1} with ${f},g\equiv0$ and $m>1$. They
 also mentioned the H\"older continuity of the weak solution of
 \eqref{eq:1} involving the external forces ${f}\in
 L^q(0,\infty\,;\,L^p(\R^n))$, $g\in
 L^\frac{q}{2}(0,\infty\,;\,L^{\frac{p}{2}}(\R^n))$ with
 $\frac2q+\frac{n}p<1$. In this paper we extend the above mentioned
 results to the case of general external forces, more specifically, we
 prove H\"older continuity for bounded weak solutions of \eqref{eq:1}.
 In addition, we obtain H\"older estimates with explicit dependence on
 the external forces ${f}$ and $g$.

 It is well-known that Harnack estimates are closely related to H\"older
 continuity of solutions (cf. Aronson-Caffarelli~\cite{MR0712265},
 DiBenedetto~\cite{MR0913961}, \cite{MR1230384} and
 DiBenedetto-Gianazza-Vespri~\cite{MR2413134}, \cite{MR2865434}). We
 remark that the porous medium equation \eqref{eq:1} is not additive, in
 particular for a solution $u$ of \eqref{eq:1} and a constant $k\in\R$,
 both $u-k$ and $k-u$ do not satisfy \eqref{eq:1}. Thus the Harnack
 inequality does not imply H\"older continuity of solutions of
 \eqref{eq:1} directly. DiBenedetto-Gianazza-Vespri~\cite{MR2413134},
 \cite{MR2865434} pointed out this fact and considered H\"older
 estimates for the singular porous medium equation, namely \eqref{eq:1}
 with $0<m<1$. It is also well-known that regularity of the gradient of
 solutions imply H\"older continuity. Gradient estimates for $p$-Laplace
 evolution equations are recently studied by
 Kinnunen-Lewis~\cite{MR1749438}, Acerbi-Mingione~\cite{MR2286632},
 Duzaar-Mingione~\cite{MR2823872} and Kuusi-Mingione~\cite{MR2813574}.

 Before stating our main theorem, we introduce weak $L^p$ spaces.
 \begin{definition}
 For a domain $\Omega\subset\R^n$ and an exponent $p>1$, a function
 $f\in L^1_{\mathrm{loc}}(\Omega)$ belongs to
 $L^p_{\mathrm{w}}(\Omega)$ if
 \[
 \|f\|_{L^p_{\mathrm{w}}(\Omega)}
 :=\sup_{K\subset\Omega\,:\,\text{compact}}
 \frac1{|K|^{1-\frac1p}}\int_K|f|\,dx<\infty.
 \]
 \end{definition}
 
\begin{remark}
 By H\"older inequality, we find $L^p(\Omega)\subset
 L^p_{\mathrm{w}}(\Omega)$. In fact, $L^p_\mathrm{w}(\Omega)$ is
 strictly larger than $L^p(\Omega)$ since $|x|^{-\frac{n}{p}}\notin
 L^p(\R^n)$ but belonging to $L^p_\mathrm{w}(\R^n)$ .
\end{remark}

Now, we state our main theorem.

 \begin{theorem}
  \label{thm:1} 
  Let $m>1$ and let $u$ be a bounded weak solution of
  \eqref{eq:1}. Assume ${f}\in L^q(0,\infty\,;\,L^p_{\mathrm{w}}(\R^n))$
  and $g\in
  L^\frac{q}{2}(0,\infty\,;\,L^\frac{p}{2}_{\mathrm{w}}(\R^n))$ for some
  $p,q>2$ satisfying $\frac2q+\frac{n}p<1$. Then, for all
  $\varepsilon>0$, the solution $u$ is uniform H\"older continuous with
  respect to $(t,x)$ in $(\varepsilon,\infty)\times\R^n$. Precisely,
  there exist constants $C,\sigma>0$ such that
  \begin{multline}
   \label{eq:6}
   |u(t,x)-u(s,y)| 
   \leq C(\|u\|_{L^\infty((0,\infty)\times\R^n)} \\
   +\|u\|^{\frac1q(1-\frac1m)}_{L^\infty((0,\infty)\times\R^n)} 
   \|{f}\|^\frac1m_{L^q(0,\infty\,;\,L^p_{\mathrm{w}}(\R^n))} 
   +\|u\|^{\frac2q(1-\frac1m)}_{L^\infty((0,\infty)\times\R^n)}
   \|g\|^\frac1m_{L^\frac{q}{2}(0,\infty\,;\,L^\frac{p}{2}_{\mathrm{w}}(\R^n))}) \\
   \times(\|u\|_{L^\infty((0,\infty)\times\R^n)}^{\frac\sigma2(1-\frac1m)}
  |t-s|^\frac{\sigma}2+|x-y|^\sigma)
  \end{multline} 
  for all $(t,x),(s,y)\in(\varepsilon,\infty)\times\R^n$, where
  $\sigma>0$ depends only on $n,m,p,q$ and $C>0$ depends only on $n,m,p,q,\varepsilon$.
 \end{theorem}

When the initial datum $u_0$ is bounded positive and the external force
$\Div{f}+g$ is bounded, then solutions of \eqref{eq:1} is bounded on
some time interval $(0,T)$ by the maximum
principle. Sugiyama-Kunii~\cite{MR2235324} and Ogawa~\cite{MR2482499}
showed the boundedness for the solution and rescaled solution of
\eqref{eq:45} hence we may apply our results for \eqref{eq:45}.

We emphasize that our estimate \eqref{eq:6} is \emph{scale
invariant}. In fact, for the parameter $M>0$, we consider the scale
transform
\begin{equation}
 \label{eq:7}
 \begin{aligned}
  t&=\frac{1}{M^{m-1}}s,\ \,&
  u_{M}(s,x)&=\frac{1}Mu(t,x),\\
  {f}_{M}(s,x)&={f}(t,x),\ \,&
  g_{M}(s,x)&=g(t,x).
 \end{aligned}
\end{equation}
Then $u_{M}$ satisfies
\[
 \partial_su_{M}-\Delta u_{M}^m
 =\Div\left(\frac{1}{M^{m}}{f}_{M}\right)
 +\left(\frac{1}{M^{m}}g_{M}\right)
\]
and Theorem \ref{thm:1} implies
  \begin{multline*}
   |u_{M}(s,y)-u_{M}(s',y')|
   \leq C\biggl(\|u_{M}\|_{L^\infty((0,\infty)\times\R^n)} \\
   +\left(\frac{1}{M^{m}}\right)^\frac1m
   \|u_{M}\|^{\frac1q(1-\frac1m)}_{L^\infty((0,\infty)\times\R^n)}
   \|{f}_{M}\|^\frac1m_{L^q(0,\infty\,;\,L^p_{\mathrm{w}}(\R^n))} \\
   +\left(\frac{1}{M^{m}}\right)^\frac1m
   \|u_{M}\|^{\frac2q(1-\frac1m)}_{L^\infty((0,\infty)\times\R^n)}
   \|g_{M}\|^\frac1m_{L^\frac{q}{2}(0,\infty\,;\,L^\frac{p}{2}_{\mathrm{w}}(\R^n))}\biggr) \\
   \times(\|u_{M}\|_{L^\infty((0,\infty)\times\R^n)}^{\frac\sigma2(1-\frac1m)}
  |s-s'|^\frac{\sigma}2+|y-y'|^\sigma).
  \end{multline*}
By the scale transform \eqref{eq:7}, we have \eqref{eq:6} hence we find
that the H\"older estimates \eqref{eq:6} is invariant for the scale
transform \eqref{eq:7}. 

Our H\"older estimates \eqref{eq:6} is some generalization for the case
of the heat equation. Actually, letting $m\rightarrow1$, we find
that the constant $C>0$ is bounded and the H\"older exponent $\sigma>0$
is away from $0$. Therefore the estimates \eqref{eq:6} implies the
well-known H\"older estimates for the case of the heat equation.

The basic strategy to prove Theorem \ref{thm:1} is to use the method of
alternative and intrinsic scaling by
DiBenedetto-Friedman~\cite{MR0783531}. Since they use the local
oscillation of solutions as the intrinsic scaling, it seems difficult to
obtain H\"older estimates of solutions.  On the other hand, we use the
local maximum of solutions as the intrinsic scaling and we make the more
exact Caccioppoli estimate.  The Caccioppoli estimate plays an important
role to show the method of alternative.  Reconstructing the iteration
argument, we obtain the H\"older estimates of solutions.

For an application, we may consider the external force as the
perturbation of solutions(cf. Ogawa-Mizuno~\cite{O-M}). Applying our
theorem, we do not need $L^p$ integrability of the external force, but
growth order of $L^2$ integral. Therefore, it is useful to study $L^2$
theory of non-linear degenerate parabolic equations. Furthermore, we can
exactly estimate the H\"older norm of solutions by the external force
and the maximum of solutions.

The paper is organized as follows. In section 2, we first give an
alternative lemma and show Theorem \ref{thm:1} using the alternative
lemma. The alternative lemma gives either the better lower bounds or the
better upper bounds of solutions. We show the lower bounds of the
solution in section 3 and the upper bounds of solution in section 4. In
appendix, we give some fundamental results of calculus which are
necessary for the proof of the main theorem.

At the end of this section, we introduce some notations. For
$\rho,M,\theta_0>0$ and $t_0\in\R$, we let open intervals
\[
 I_\rho(t_0)=(t_0-\rho^2,t_0),\quad
 I_\rho^{\theta_0}(t_0)=\left(t_0-\frac{\theta_0}2\rho^2,t_0\right)
\]
and
\[
 I_{\rho,M}(t_0)=\left(t_0-\frac{\rho^2}{M^{1-\frac1m}},t_0\right),\quad
 I_{\rho,M}^{\theta_0}(t_0)=
 \left(t_0-\frac{\theta_0}2\frac{\rho^2}{M^{1-\frac1m}},t_0\right).
\]
For $\rho>0$ and $x_0\in\R^n$, we denote the $n$-dimensional open ball
with radius $\rho$ and center $x_0$ by $B_\rho(x_0)$. We define
parabolic cylinders $Q_\rho(t_0,x_0),\ Q^{\theta_0}_\rho(t_0,x_0)$ and
modified parabolic cylinders $Q_{\rho,M}(t_0,x_0),\
Q^{\theta_0}_{\rho,M}(t_0,x_0)$ by
\[
 Q_\rho(t_0,x_0)=I_\rho(t_0)\times
 B_\rho(x_0),\quad Q^{\theta_0}_\rho(t_0,x_0)=I^{\theta_0}_\rho(t_0)\times
 B_\rho(x_0)
\]
and
\[
 Q_{\rho,M}(t_0,x_0)=I_{\rho,M}(t_0)\times
 B_\rho(x_0),\quad
 Q^{\theta_0}_{\rho,M}(t_0,x_0)=I^{\theta_0}_{\rho,M}(t_0)\times
 B_\rho(x_0).
\]
We often abbreviate the center of parabolic cylinders $(t_0,x_0)$. We
denote the $n$-dimensional Lebesgue measure by $m_n$. A function space
$L^q(I_{\rho,M};L^p_{\mathrm{w}}(B_\rho))$ is abbreviated to
$L^q(L^p_{\mathrm{w}})(Q_{\rho,M})$ and another function spaces are also
same. For a function $f$ on a set $A$, we denote the oscillation of $f$
in $A$ by $\osc_{A}f:=\sup_Af-\inf_Af$.  We denote the positive part of
$f$ and the negative part of $f$ by $f_+:=\max\{0,f\}$ and
$f_-:=\max\{0,-f\}$, respectively. We remark that a superscript plus or
minus are different of the positive part or the negative part.  For a
constant $k\in\R$ and a function $f$ on a set $\Omega$, we let
\[
 \{f>k\}:=\{x\in\Omega:f(x)>k\}
\]
and other level sets such as $\{f<k\}$ are defined in a similar manner.
We put $\sigma_0=1-\frac2q-\frac{n}p$ and $h(\rho,M,\omega):=
\big\|{f}\big\|^2_{L^q(L^p_\mathrm{w})(Q_{\rho,M})}
+\omega\|g\|_{L^\frac{q}2(L^{\frac{p}2}_\mathrm{w})(Q_{\rho,M})}$.
We denote a constant depending on $m,\beta,\dots$ by
$C(m,\beta,\dots)$. The same letter $C$ will be used to denote
different constants. We use subscript numbers if we consider the
relation between the constants.
For a open interval $(a,b)\subset\R$ and a open ball
$B_\rho(x_0)\subset\R^n$, we call $\eta=\eta(t,x)$ a cut-off function in
$Q=(a,b)\times B_\rho(x_0)$ if $\eta\in C^\infty(Q)$ satisfies
\[
  \eta(t,x)\equiv0\quad a\leq t\leq b,\quad x\in\partial B_\rho(x_0)
  \quad\text{and}\quad
  \eta(a,x)\equiv0\quad x\in B_\rho(x_0).
\]

\begin{figure}
 \begin{center}
 \includegraphics[width=5cm]{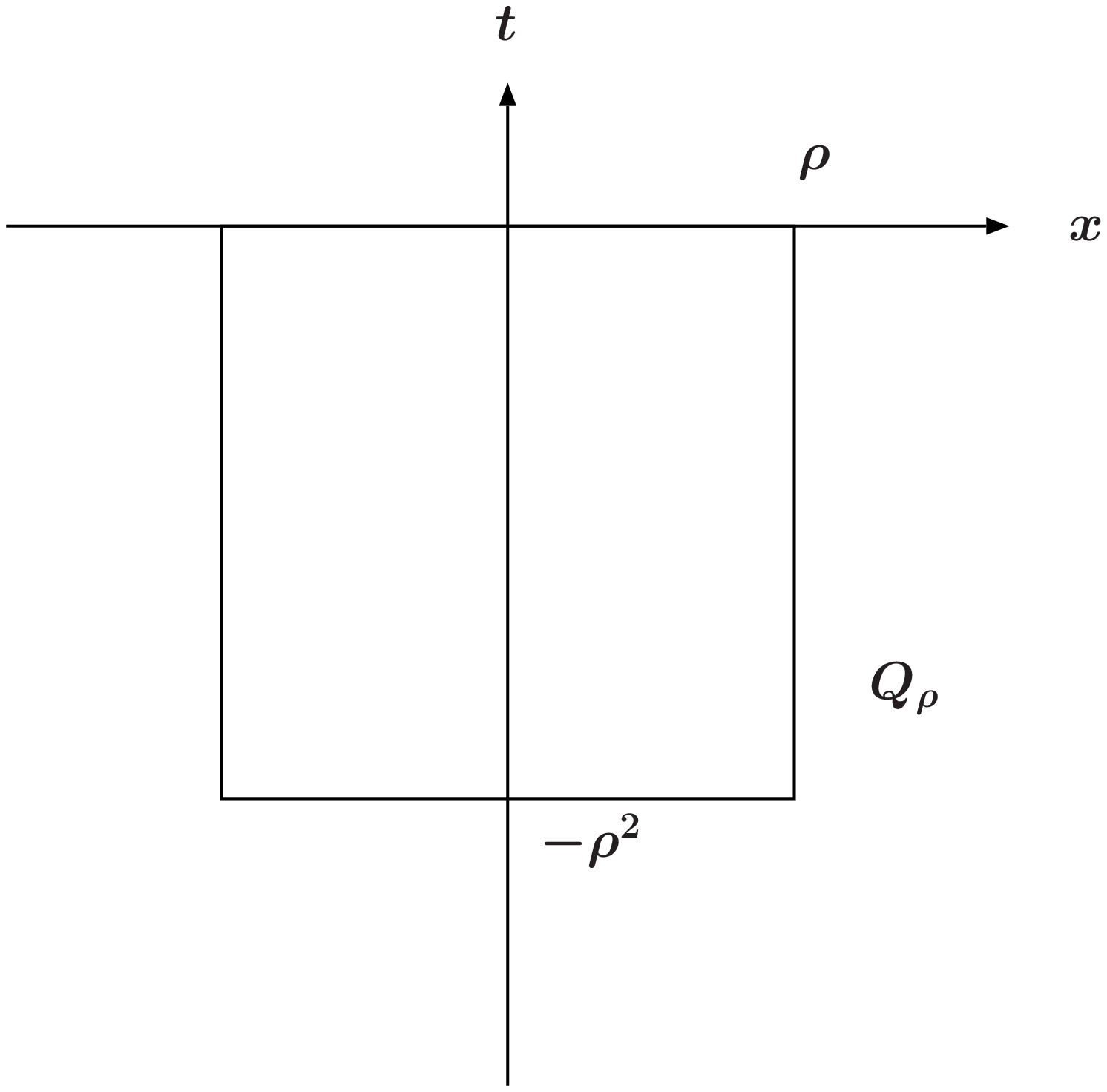}
  \hspace{1cm}
 \includegraphics[width=5cm]{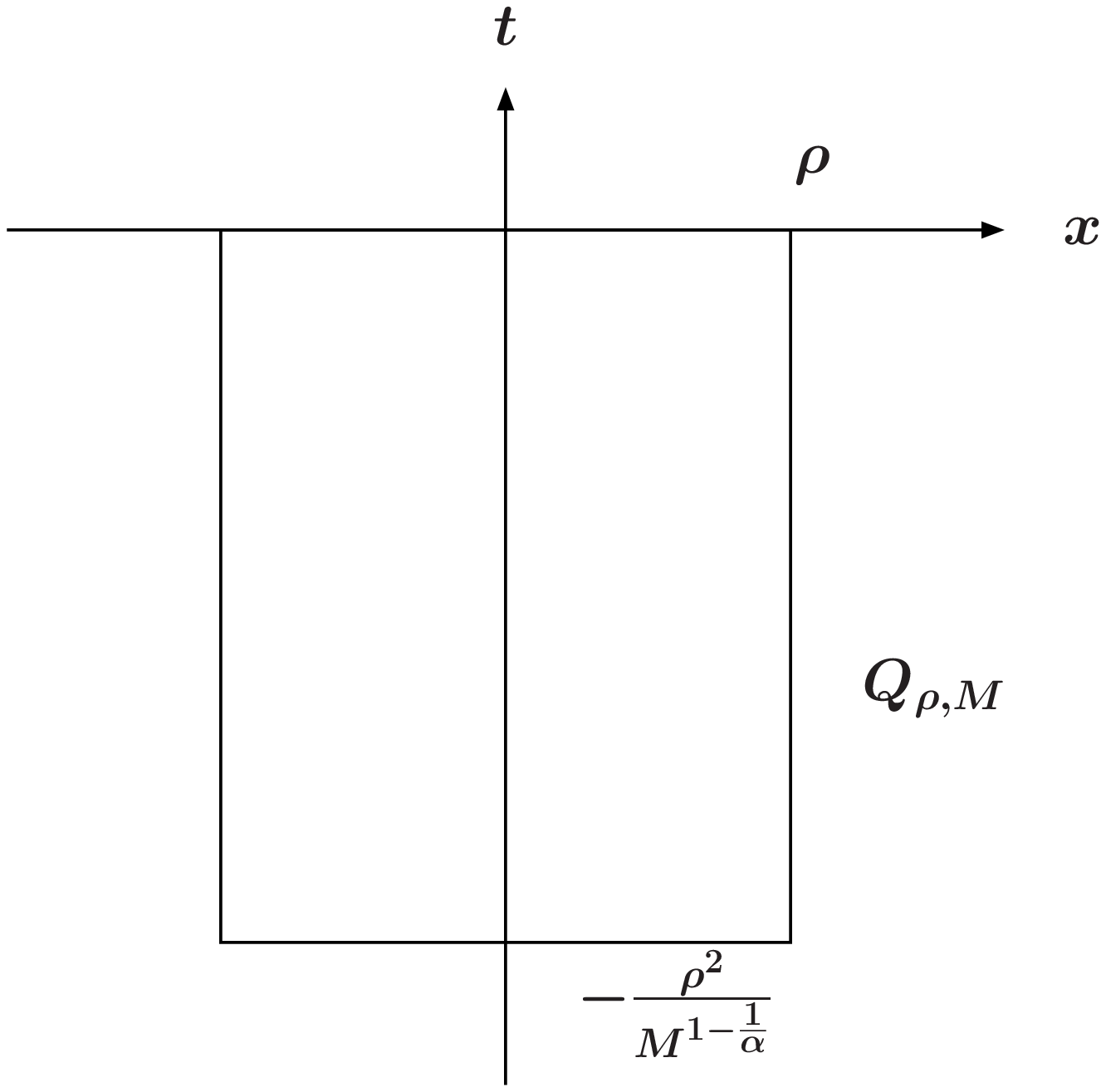}
  \caption{the usual parabolic cylinder and the modified parabolic cylinder}
 \end{center}
\end{figure}



\section{Alternative lemma and proof of the main theorem}
We hereafter replace $u^m$ by $u$ and we consider the following
equation:
\begin{equation}
 \label{eq:3}
  \partial_tu^\frac1m-\Delta u=-\Div {f}+g.
\end{equation}
Let $M$ and $\omega$ be an approximated supremum and oscillation of the
weak solution $u$ of \eqref{eq:3}, namely
\begin{equation}
 \label{eq:4}
  \sup_{Q_{\rho,M}(t_0,x_0)}u\leq M\leq 3 \sup_{Q_{\rho,M}(t_0,x_0)}u,
\end{equation}
and
\begin{equation}
 \label{eq:5}
  \frac34\omega\leq\osc_{Q_{\rho,M}(t_0,x_0)}u\leq \omega.
\end{equation}

\begin{lemma}[alternative lemma]
 \label{lem:1}
 Let us assume \eqref{eq:4} and \eqref{eq:5}. Then there exist constants
 $0<\theta_0,\eta_0<1$ and $\delta_0>0$ depending only on $n,m,p,q$
 such that for all $\rho>0$ satisfying
 $\rho^{\sigma_0}\leq\delta_0\omega M^{-\frac1q({1-\frac1m})}
 h(\rho,M,\omega)^{-\frac12}$,
 we obtain the following estimates:
 \begin{enumerate}[(i)]
  \item \label{item:1}
	Lower bounds. If
	\begin{equation*}
	  m_{n+1}\left(
		  Q_{\rho,M}(t_0,x_0)
		  \cap\biggl\{u<\inf_{Q_{\rho,M}(t_0,x_0)}u+\frac\omega2\biggr\}
		      \right) 
	 \leq\theta_0m_{n+1}\Bigl(Q_{\rho,M}(t_0,x_0)\Bigr),
	\end{equation*}	 
	then
	\[
	 u(t,x)\geq\inf_{Q_{\rho,M}(t_0,x_0)}u+\eta_0\omega\quad\text{for}\ 
	(t,x)\in Q_{\frac\rho2,M}(t_0,x_0);
	\]
  \item \label{item:2}
	Upper bounds. If
	\begin{equation*}
	 m_{n+1}\left(
		 Q_{\rho,M}(t_0,x_0)
		 \cap\biggl\{u<\inf_{Q_{\rho,M}(t_0,x_0)}u+\frac\omega2\biggr\}
		\right) 
	 >\theta_0m_{n+1}\Bigl(Q_{\rho,M}(t_0,x_0)\Bigr),
	\end{equation*}	
	then
	\[
	u(t,x)\leq\sup_{Q_{\rho,M}(t_0,x_0)}u-\eta_0\omega\quad\text{for}\ 
	(t,x)\in Q^{\theta_0}_{\frac\rho2,M}(t_0,x_0).
	\]
 \end{enumerate}
\end{lemma}
We will prove part (\ref{item:1}), which is Proposition \ref{prop:1} in
Section \ref{sec:3} and part (\ref{item:2}), which is Proposition
\ref{prop:2} in Section \ref{sec:4}. According to Lemma \ref{lem:1}, we
obtain
\begin{equation}
 \label{eq:8}
  \osc_{Q^{\theta_0}_{\frac\rho2,M}(t_0,x_0)}u
  \leq\osc_{Q_{\rho,M}(t_0,x_0)}u-\eta_0\omega
  \leq(1-\eta_0)\omega
\end{equation}
provided
$\rho^{\sigma_0}\leq\delta_0\omega M^{-\frac1q({1-\frac1m})}
 h(\rho,M,\omega)^{-\frac12}$. We remark that we may take $\eta_0$ as
 small as we want since we obtain by \eqref{eq:8}
 \[
  \osc_{Q^{\theta_0}_{\frac\rho2,M}(t_0,x_0)}u
  \leq(1-\eta_0)\omega
  \leq (1-\eta)\omega
 \]
for any $0<\eta<\eta_0$.

\begin{remark}
 We explain an advantage to use the modified parabolic cylinder. For
 $\rho\ll1$ and $M>0$, we consider
 \[
  \partial_tu^\frac1m-\Delta u=-\Div {f}+g
 \quad \text{in}
 \ Q_{\rho,M}.
 \]
 Introducing the scale transform
 \begin{align*}
  t&=\frac{\rho^2}{M^{1-\frac1m}}s, &
  x&=\rho y, &&\\
  u_{\rho,M}(s,y)&=\frac1Mu(t,x), 
  & {f}_{\rho,M}(s,y)&={f}(t,x), &
  g_{\rho,M}(s,y)&=g(t,x), 
 \end{align*}
 we obtain 
 \begin{equation}
   \label{eq:10}
  \partial_su_{\rho,M}^\frac1m-\Delta_y u_{\rho,M}
 =-\Div\Bigl(\frac\rho{M}{f}_{\rho,M}\Bigr)
 +\frac{\rho^2}{M}g_{\rho,M}
 \quad \text{in}
 \ Q_1.
 \end{equation}
 Since $M$ can be regarded as the supremum of $u$ on $Q_{\rho,M}$ by the
 assumption \eqref{eq:4}, we may consider \eqref{eq:10} as the uniformly
 parabolic equation. Furthermore, in view of
 \[
 \begin{split}
  \bigg\|\frac{\rho}M{f}_{\rho,M}
  \bigg\|^2_{L^q(L^p_\mathrm{w})(Q_1)}
  &=\rho^{2(1-\frac2q-\frac{n}p)}M^{-2+\frac2q(1-\frac1m)}
  \|{f}\|^2_{L^q(L^p_\mathrm{w})(Q_{\rho,M})}, \\
  \bigg\|\frac{\rho^2}Mg_{\rho,M}
  \bigg\|_{L^\frac{q}{2}(L^\frac{p}{2}_\mathrm{w})(Q_1)}
  &=\rho^{2(1-\frac2q-\frac{n}p)}M^{-1+\frac2q(1-\frac1m)}
  \|g\|_{L^\frac{q}{2}(L^\frac{p}{2}_\mathrm{w})(Q_{\rho,M})}, \\
 \end{split} 
\]
 the inequality $1-\frac2q-\frac{n}p>0$ is the sufficient condition to
 ignore the external force.
\end{remark}

 We now show Theorem \ref{thm:1} by temporary admitting Lemma
 \ref{lem:1}. We put $Q={(0,\infty)\times\R^n}$, $M_0=\sup_{Q}u$ and
 $\omega_0=M_0$.  Let $\theta_0,\,\delta_0$ and $\eta_0$ be as in Lemma
 \ref{lem:1}. We choose $0<\rho_0<\varepsilon$ satisfying
 \[
 \rho_0^{\sigma_0}
 \leq\delta_0\omega_0M^{-\frac1q(1-\frac1m)}_0
 (\|f\|^2_{L^q(0,\infty\,;\,L^p_{\mathrm{w}}(\R^n))}
 +\omega_0
 \|g\|_{L^\frac{q}{2}(0,\infty\,;\,L^\frac{p}{2}_{\mathrm{w}}(\R^n))})^{-\frac12}. 
 \]
 For $(t_0,x_0)\in(0,\infty)\times\R^n$, We
 denote $Q_0=Q_{\rho_0,M_0}(t_0,x_0)\,,\,\mu^+_0=\sup_{Q_0}u$ and
 $\mu^-_0=\inf_{Q_0}u$. Then, we find
\[
 \left\{
 \begin{aligned}
 \osc_{Q_0}u&\leq\omega_0, \\ 
  \sup_{Q_0}u&\leq \sup_{Q}u\leq M_0, \\
  \rho_{0}^{\sigma_0}
  &\leq \delta_0\omega_0M_0^{-\frac1q(1-\frac1m)}
  h(\rho_0,M_0,\omega_0)^{-\frac12}.
 \end{aligned}
 \right.
\]
We choose
\[
 r_0:=\min\biggl\{(1-\eta_0)^\frac{1}{\sigma_0},\,
 \frac12\Bigl(\frac13\Bigr)^{\frac12(1-\frac1m)}
 \Bigl(\frac{\theta_0}{2}\Bigr)^\frac12\biggr\}
\]
and choose sequences as follows: For $j\in\N$, 
\begin{equation}
 \label{eq:40}
  \begin{aligned}
  \omega_j&:=(1-\eta_0)\omega_{j-1},& 
  \rho_j&:=r_0\rho_{j-1}, \\
  M_j&:=\max\{\mu^+_{j-1},\,\omega_j\},&
  Q_{j}&:=Q_{\rho_j,M_j}(t_0,x_0), \\    
  \mu^+_j&:=\sup_{Q_j}u,&
  \mu^-_j&:=\inf_{Q_j}u.
 \end{aligned} 
\end{equation}
Using Lemma \ref{lem:1}, we obtain the following oscillation estimates.

\begin{lemma}
 \label{lem:2} 
 Let $\{\omega_j,\rho_j,M_j,Q_j\}_{j=0}^\infty$ is defined
 by the above \eqref{eq:40}. Then for $0<\delta_0<1$ defined in Lemma
 \ref{lem:1} and for $j\in\N$, we obtain
 \begin{equation}
  \label{eq:9}
   \left\{
    \begin{aligned}
     \osc_{Q_{j}}u&\leq\omega_j, \\
     \sup_{Q_j}u&\leq \sup_{Q_{j-1}}u\leq M_j, \\ 
     \rho_{j}^{\sigma_0}
     &\leq \delta_0\omega_jM_j^{-\frac1q(1-\frac1m)}
     h(\rho_j,M_j,\omega_j)^{-\frac12}. 
    \end{aligned}
   \right.
\end{equation}
\end{lemma}

\begin{proof}%
 [Proof of Lemma \ref{lem:2}]
 By the definition of $M_j$, we obtain
 $\sup_{Q_j}u\leq \sup_{Q_{j-1}}u\leq M_j$. Since $r_0\leq
 (1-\eta_0)^\frac{1}{\sigma_0}$ and definition of $\omega_j$, we find  
 \[
 \rho_{j}^{\sigma_0}
 \leq \delta_0\omega_jM_j^{-\frac1q(1-\frac1m)}
 h(\rho_j,M_j,\omega_j)^{-\frac12}. 
 \]
 We show $\osc_{Q_{j}}u\leq\omega_j$.
 
 To show $\osc_{Q_{j}}u\leq\omega_j$, we make induction. First we
 consider the case $j=1$.  Either if $\osc_{Q_0}u\leq\frac34\omega_0$,
 then we find $Q_1\subset Q_0$ since
 $r_0\leq(\frac34)^{\frac12(1-\frac1m)}$ and
 \[
 \frac{M_1}{M_0}
 \geq\frac{\omega_1}{M_0}
 =(1-\eta_0)\geq\frac34.
 \]
 For this reason, we obtain
 \[
 \osc_{Q_1}u\leq
 \osc_{Q_0}u\leq\frac34\omega_0\leq(1-\eta_0)\omega_0=\omega_1.
 \]
 Otherwise, if
 $\frac34\omega_0\leq\osc_{Q_0}u\leq\omega_0$, we obtain
 $M_0=\omega_0\leq\frac43\mu^+_0$. Applying Lemma \ref{lem:1}, we
 find
 \[
  \osc_{Q^{\theta_0}_{\frac{\rho_0}{2},M_0}(t_0,x_0)}u\leq(1-\eta_0)\omega_0.
 \]
 Since $r_0\leq\frac12\Bigl(\frac13\Bigr)^{\frac12(1-\frac1m)}
 \Bigl(\frac{\theta_0}{2}\Bigr)^\frac12$, we have
 $Q_1\subset Q^{\theta_0}_{\frac{\rho_0}{2},M_0}(t_0,x_0)
 \subset Q_0$ and hence
 \[
 \osc_{Q_1}u
 \leq\osc_{Q^{\theta_0}_{\frac{\rho_0}{2},M_0}(t_0,x_0)}u
 \leq(1-\eta_0)\omega_0=\omega_1.
 \]
 In either case, we obtain \eqref{eq:9} for $j=1$.  Next we assume
 \eqref{eq:9} for $j\leq k$ and we show for $j=k+1$ using the following
 inequality:
 \begin{equation}
  \label{eq:41}  
   \mu^+_{k-1}\leq\max\biggl\{\frac3{2(1-\eta_0)}\omega_k\,,\,3\mu^+_k\biggr\}.
 \end{equation}
 To show \eqref{eq:41}, we consider the case
 $\mu^-_{k-1}\leq\frac13\mu^+_{k-1}$ first. Then
 \[
 \mu^+_{k-1}
 \leq\osc_{Q_{k-1}}u+\mu_{k-1}^-
 \leq\omega_{k-1}+\frac13\mu^+_{k-1}
 \]
 and hence $\mu^+_{k-1}\leq
 \frac32\omega_{k-1}=\frac{3}{2(1-\eta_0)}\omega_k$.  For the other
 case, namely if $\mu^-_{k-1}>\frac13\mu^+_{k-1}$, then we have
 $\mu^+_{k-1}<3\mu^-_{k-1}\leq 3\mu^-_{k}\leq3\mu^+_{k}$ and we obtain
 \eqref{eq:41}.

 We show \eqref{eq:9} for $j=k+1$. First we consider the case
 $\osc_{Q_k}u\leq\frac34\omega_k$ and we show $Q_{k+1}\subset Q_k$. Either
 if $M_k=\omega_k$, then
 \[
 \frac{M_{k+1}}{M_k}
 =\frac{M_{k+1}}{\omega_k}
 \geq\frac{(1-\eta_0)\omega_k}{\omega_k}
 =(1-\eta_0)\geq\frac34.
 \]
 Since $r_0\leq(\frac34)^{\frac12(1-\frac1m)}$, we obtain
 $Q_{k+1}\subset Q_k$.
 Otherwise, if $M_k=\mu^+_{k-1}$, we obtain by \eqref{eq:41}
 \[
 \begin{split}
  \frac{M_{k+1}}{M_k}
  =\frac{M_{k+1}}{\mu^+_{k-1}}
  &\geq\frac{M_{k+1}}{\max\Bigl\{\frac3{2(1-\eta_0)}\omega_k\,,\,
  3\mu^+_k\Bigr\}} \\
  &\geq
  \frac{1}{\max\Bigl\{\frac3{2(1-\eta_0)}\frac{\omega_k}{M_{k+1}}\,,\,
  \frac{3\mu^+_k}{M_{k+1}}\Bigr\}}
  \\
  &\geq
  \frac{1}{\max\Bigl\{\frac3{2(1-\eta_0)^2}\,,\,
  3\Bigr\}}
  \geq\frac13.
 \end{split} 
 \]
 Since $r_0\leq(\frac13)^{\frac12(1-\frac1m)}$, we have
 $Q_{k+1}\subset Q_k$. In either case, we have $Q_{k+1}\subset Q_k$ and
 hence
 \[
 \osc_{Q_{k+1}}u
 \leq \osc_{Q_k}u
 \leq \frac34\omega_k\leq\omega_{k+1.}
 \]
 Second we consider the case 
 $\frac34\omega_k\leq\osc_{Q_k}u\leq\omega_k$. Since
 $\omega_k\leq\frac43\mu^+_k$, we obtain
 \[
 \mu^+_{k-1}
 \leq\max\biggl\{\frac3{2(1-\eta_0)}\omega_k\,,\,
 3\mu^+_k\biggr\}
 \leq \max\biggl\{\frac2{(1-\eta_0)}\mu^+_k\,,\,
 3\mu^+_k\biggr\}
 \leq3\mu^+_k
 \]
 and hence
 \[
  M_k\leq\max\biggl\{\frac43\mu^+_k\,,\,3\mu^+_k\biggr\}
 \leq3\mu^+_k.
 \]
 Hence we may apply Lemma \ref{lem:1} and we obtain
 \[
  \osc_{Q^{\theta_0}_{\frac{\rho_k}{2},M_k}(t_0,x_0)}u\leq(1-\eta_0)\omega_k
 =\omega_{k+1}.
 \]
 Since $r_0\leq \frac12\left(\frac13\right)^{\frac12(1-\frac1m)}
 \left(\frac{\theta_0}{2}\right)^\frac12$ and
 $\frac{M_{k+1}}{M_k}
 \geq\frac{\mu^+_k}{3\mu^+_k}=\frac13$, 
 we have $Q_{k+1}\subset Q^{\theta_0}_{\frac{\rho_k}{2},M_k}(t_0,x_0)$ and
 hence 
 \[
 \osc_{Q_{k+1}}u
 \leq\osc_{Q^{\theta_0}_{\frac{\rho_k}{2},M_k}(t_0,x_0)}u
 \leq\omega_{k+1}.  
 \]
\end{proof}

 \begin{proof}%
  [Proof of Theorem \ref{thm:1}] 
  Remarking that $M_j\geq M_{j+1}$ for $j\in\N$, we have by Lemma
  \ref{lem:2}
  \[
  \osc_{Q_{\rho_j,M_0}(t_0,x_0)}u\leq\osc_{Q_j}u\leq\omega_{j}.
  \]
  We choose $0<\sigma<1$ satisfying $r_0^\sigma\geq 1-\eta_0$. Then
 we obtain
 \[
 \osc_{Q_{\rho_j,M_0}(t_0,x_0)}u
 \leq (1-\eta_0)^j\omega_0
 =\omega_0\biggl(\frac{\rho_j}{\rho_0}\biggr)^\sigma.
 \]
 For $\rho\leq\rho_0$, there exists $k\in\N_0$ such that
 $\rho_k\leq\rho\leq\rho_{k-1}$ and hence
 \[
  \osc_{Q_{\rho,M_0}(t_0,x_0)}u
 \leq \omega_0\biggl(\frac{\rho_{k-1}}{\rho_0}\biggr)^\sigma
 = \omega_0r_0^{-\sigma}\biggl(\frac{\rho_{k}}{\rho_0}\biggr)^\sigma
 \leq M_0r_0^{-\sigma}\biggl(\frac{\rho}{\rho_0}\biggr)^\sigma.
 \]
 Taking $\rho_0>0$ as 
  \[
  \rho_0^{\sigma_0}
  =\delta_0\omega_0M^{-\frac1q(1-\frac1m)}_0
  (\|{f}\|^2_{L^q(L^p_{\mathrm{w}})(Q)} +\omega_0
  \|g\|_{L^\frac{q}{2}(L^\frac{p}{2}_{\mathrm{w}})(Q)})^{-\frac12},
  \]
 we find
  \begin{equation}
   \osc_{Q_{\rho,M_0}(t_0,x_0)}u 
   \leq CM_0^{1-\frac{\sigma}{\sigma_0}}
    M_0^{\frac{\sigma}{q\sigma_0}(1-\frac1m)} 
    (\|{f}\|^2_{L^q(L^p_{\mathrm{w}})(Q)}
    +\omega_0
    \|g\|_{L^\frac{q}{2}(L^\frac{p}{2}_{\mathrm{w}})(Q)})^\frac{\sigma}{2\sigma_0}\rho^\sigma
  \end{equation} 
 for $\rho\leq\rho_0$ where the constant $C$ depends only on
 $n,m,p$ and $q$. Furthermore, if $\rho>\rho_0$, then
  \begin{equation*}
   \osc_{Q_{\rho,M_0}(t_0,x_0)}u
    \leq M_0\biggl(\frac{\rho}{\rho_0}\biggr)^\sigma 
   \leq CM_0^{1-\frac{\sigma}{\sigma_0}}
    M_0^{\frac{\sigma}{q\sigma_0}(1-\frac1m)}
    (\|{f}\|^2_{L^q(L^p_{\mathrm{w}})(Q)}
    +\omega_0
    \|g\|_{L^\frac{q}{2}(L^\frac{p}{2}_{\mathrm{w}})(Q)})^\frac{\sigma}{2\sigma_0}\rho^\sigma.
  \end{equation*}
 Therefore, we find
 \begin{equation*}
  \begin{split}
   \osc_{Q_{\rho,M_0}(t_0,x_0)}u 
   &\leq C
   (M_0+M_0^{\frac1q(1-\frac1m)}
   (\|{f}\|^2_{L^q(L^p_{\mathrm{w}})(Q)}
   +M_0
   \|g\|_{L^\frac{q}{2}(L^\frac{p}{2}_{\mathrm{w}})(Q)})^\frac12
   \rho^\sigma \\
   &\leq C
   (M_0+M_0^{\frac1q(1-\frac1m)}
   \|{f}\|_{L^q(L^p_{\mathrm{w}})(Q)}
   +M_0^{\frac12+\frac1q(1-\frac1m)}
   \|g\|^\frac12_{L^\frac{q}{2}(L^\frac{p}{2}_{\mathrm{w}})(Q)})
   \rho^\sigma \\
   &\leq C
   (M_0+M_0^{\frac1q(1-\frac1m)}
   \|{f}\|_{L^q(L^p_{\mathrm{w}})(Q)}
   +M_0^{\frac2q(1-\frac1m)}
   \|g\|_{L^\frac{q}{2}(L^\frac{p}{2}_{\mathrm{w}})(Q)})
   \rho^\sigma 
  \end{split} 
 \end{equation*} 
 and proof of Theorem \ref{thm:1} is complete.
 \end{proof}



\section{Proof of Lower bounds (\ref{item:1}) of Lemma \ref{lem:1}}
\label{sec:3}
Without loss of generality, we assume $t_0=0$ by using the parallel
translation. We omit the center of ball $x_0$. We hereafter write
$\mu^+=\sup_{Q_{\rho,M}}u,\,\ \mu^-=\inf_{Q_{\rho,M}}u$. 

In this section, we prove Lower bounds in Lemma \ref{lem:1}. More
precisely, we show the following proposition:

\begin{proposition}[First alternative]
 \label{prop:1}
 Let $\rho>0$ satisfying 
 \[
 \rho^{\sigma_0}\leq \omega
 {M^{-\frac1q(1-\frac1m)}}h(\rho,M,\omega)^{-\frac12}.
 \]
 Assume inequalities \eqref{eq:4} and \eqref{eq:5}. Then there exists
 $0<\theta_0<1$ depending only on $n,m,p,q$ such that if
 \begin{equation*}
  m_{n+1}\left(Q_{\rho,M}\cap\Bigl\{u<\mu^-+\frac\omega2\Bigr\}\right)
   \leq\theta_0m_{n+1}\bigl(Q_{\rho,M}\bigr),
 \end{equation*}
 then
 \begin{equation*}
  u(t,x)\geq\mu^-+\frac\omega4\quad\text{for}\ (t,x)\in Q_{\frac\rho2,M}.
  \end{equation*}
\end{proposition}

To show the lower bounds, the following Caccioppoli estimate plays an
important role.

\begin{lemma}[the Caccioppoli estimate for sub-level sets]
 \label{lem:3}
 Let $\eta=\eta(t,x)$ be a cut-off function in $Q_{\rho,M}$. For
 $\mu^-<k<\mu^-+\frac12\omega$, there exists a constant $C>0$ depending
 only on $m$ such that
 \begin{multline}
  \label{eq:11}
  \sup_{t\in I_{\rho,M}}
  \int_{B_\rho}(u(t)-k)^2_-\eta^2\,dx
  +(\mu^+)^{1-\frac1m}\iint_{Q_{\rho,M}}|\nabla(u-k)_-|^2\eta^2\,dtdx\\
  \leq C\biggl\{
  \omega\iint_{Q_{\rho,M}}(u-k)_-\eta\partial_t\eta\,dtdx
  +(\mu^+)^{1-\frac1m}\iint_{Q_{\rho,M}}(u-k)_-^2|\nabla\eta|^2\,dtdx \\
  +(\mu^+)^{1-\frac1m}h(\rho,M,\omega)\biggl(\int_{I_{\rho,M}}
  m_n\Bigl(B_\rho\cap\{u(t)<k\}\Bigr)^{q'(\frac12-\frac1p)}\,dt\biggr)^{\frac2{q'}}
  \biggr\},
 \end{multline}
 where $\frac12=\frac1q+\frac1{q'}$.
\end{lemma}

\begin{proof}%
 Testing a function $-(u-k)_-\eta^2$ in \eqref{eq:3}, we obtain
\begin{multline*}
 \frac1m\iint_{Q_{\rho,M}}
 \partial_t\biggl(\int_0^{(u-k)_-}(k-\xi)^{\frac1m-1}\xi\,d\xi\biggr)\eta^2\,dtdx \\
 +\iint_{Q_{\rho,M}}\nabla(u-k)_-\cdot\nabla\{(u-k)_-\eta^2\}\,dtdx \\
 =-\iint_{Q_{\rho,M}}{f}\cdot\nabla\{(u-k)_-\eta^2\}\,dtdx
 -\iint_{Q_{\rho,M}}g(u-k)_-\eta^2\,dtdx.
\end{multline*}
 By the integration by parts and the Young inequality, we obtain
 \begin{multline}
  \label{eq:12}
  \frac1m\sup_{t\in I_{\rho,M}}\int_{B_\rho}
  \biggl(\int_0^{(u(t)-k)_-}(k-\xi)^{\frac1m-1}\xi\,d\xi\biggr)
  \eta^2(t)\,dx 
  +\frac14\iint_{Q_{\rho,M}}|\nabla(u-k)_-|^2\eta^2\,dtdx \\
  \leq\frac1m\iint_{Q_{\rho,M}}
  \biggl(\int_0^{(u-k)_-}(k-\xi)^{\frac1m-1}\xi\,d\xi\biggr)
  \partial_t\eta^2\,dtdx \\
  +3\iint_{Q_{\rho,M}}(u-k)_-^2|\nabla\eta|^2\,dtdx \\
  +2\iint_{Q_{\rho,M}\cap\{u<k\}}|{f}|^2\eta^2\,dtdx
  +\iint_{Q_{\rho,M}\cap\{u<k\}}|g|(u-k)_-\eta^2\,dtdx.
 \end{multline}
 We estimate the 1st term of the left-hand side of
 \eqref{eq:12}. Since \eqref{eq:5} and 
 $k\leq\mu^-+\frac\omega2 \leq\mu^+-\osc_{Q_{\rho,M}}u+\frac\omega2
 \leq\mu^+$, we have
 \[
 (k-\xi)^{\frac1m-1}
 \geq k^{\frac1m-1}
 \geq (\mu^+)^{\frac1m-1}\quad\text{for}\ \,\xi\geq0
 \]
 and hence
 \begin{equation}
  \label{eq:13}
   \begin{split}
 &\quad
 \frac1{2m}\sup_{t\in I_{\rho,M}}\int_{B_\rho}
 (u(t)-k)_-^2\eta^2(t)\,dx 
    +\frac14(\mu^+)^{1-\frac1m}\iint_{Q_{\rho,M}}|\nabla(u-k)_-|^2\eta^2
 \,dtdx \\
  &\leq\frac1m(\mu^+)^{1-\frac1m}\iint_{Q_{\rho,M}}
  \biggl(\int_0^{(u-k)_-}(k-\xi)^{\frac1m-1}\xi\,d\xi\biggr)
  \partial_t\eta^2\,dtdx \\
 &\quad+3(\mu^+)^{1-\frac1m}\iint_{Q_{\rho,M}}(u-k)_-^2|\nabla\eta|^2\,dtdx 
    +2(\mu^+)^{1-\frac1m}\iint_{Q_{\rho,M}\cap\{u<k\}}|{f}|^2\eta^2\,dtdx \\
 &\quad+(\mu^+)^{1-\frac1m}\iint_{Q_{\rho,M}\cap\{u<k\}}|g|(u-k)_-\eta^2\,dtdx \\
  &=:I_1+I_2+I_3+I_4.
   \end{split} 
\end{equation}
 
We estimate $I_3$ and $I_4$. By the definition of the weak $L^p$ space
and by the H\"older inequality, we have
 \[
 \begin{split}
  \iint_{Q_{\rho,M}\cap\{u<k\}}|{f}|^2\eta^2\,dtdx 
  &=\int_{I_{\rho,M}}\,dt
  \int_{B_\rho\cap\{u(t)<k\}}|{f}|^2\,dx \\
  &\leq\int_{I_{\rho,M}}
  \big\||{f}(t)|^2\big\|_{L^{\frac{p}{2}}_{\mathrm{w}}(B_\rho)}
  m_n\Bigl(B_\rho\cap\{u(t)<k\}\Bigr)^{1-\frac2p}\,dt \\
  &\leq \big\||{f}|^2\big\|_{L^{\frac{q}{2}}(L^{\frac{p}{2}}_{\mathrm{w}})(Q_{\rho,M})}
  \biggl(\int_{I_{\rho,M}}
  m_n\Bigl(B_\rho\cap\{u(t)<k\}\Bigr)^{q'(\frac12-\frac1p)}\,dt\biggr)^\frac2{q'},
 \end{split}
 \]
 and
 \[
 \begin{split}
  \iint_{Q_{\rho,M}\cap\{u<k\}}|g|(u-k)_-\eta^2\,dtdx 
  &=\frac\omega2\int_{I_{\rho,M}}\,dt
  \int_{B_\rho\cap\{u(t)<k\}}|g|\,dx \\
  &\leq\frac\omega2\int_{I_{\rho,M}}
  \|g(t)\|_{L^{\frac{p}{2}}_{\mathrm{w}}(B_\rho)}
  m_n\Bigl(B_\rho\cap\{u(t)<k\}\Bigr)^{1-\frac2p}\,dt \\
  &\leq \frac\omega2\|g\|_{L^{\frac{q}{2}}
  (L^{\frac{p}{2}}_{\mathrm{w}})(Q_{\rho,M})}
  \biggl(\int_{I_{\rho,M}}
  m_n\Bigl(B_\rho\cap\{u(t)<k\}\Bigr)^{q'(\frac12-\frac1p)}\,dt
  \biggr)^\frac2{q'}.
 \end{split}
 \]
 Therefore
 \begin{equation}
  \label{eq:14}
   I_3+I_4 
  \leq 2(\mu^+)^{1-\frac1m}
   h(\rho,M,\omega)
  \biggl(\int_{I_{\rho,M}}
  m_n\Bigl(B_\rho\cap\{u(t)<k\}\Bigr)^{q'(\frac12-\frac1p)}\,dt
  \biggr)^\frac2{q'}.
 \end{equation}  
 We estimate $I_1$. Since
 \[
 \begin{split}
  \int_0^{(u-k)_-}(k-\xi)^{\frac1m-1}\xi\,d\xi
  &\leq-m(u-k)_-\int_0^{(u-k)_-}\frac{\partial}{\partial\xi}(k-\xi)^\frac1m\,d\xi \\
  &=m(u-k)_-[k^\frac1m-(k-(u-k)_-)^\frac1m],
 \end{split} 
 \]
 we have
 \[
 \begin{split}
   I_1
  &\leq(\mu^+)^{1-\frac1m}
  \iint_{Q_{\rho,M}}[k^\frac1m-(k-(u-k)_-)^\frac1m](u-k)_
 -\partial_t\eta^2\,dtdx \\
  &\leq(\mu^+)^{1-\frac1m}
  \iint_{Q_{\rho,M}}\biggl[\biggl(\mu^-+\frac\omega2\biggr)^\frac1m-(\mu^-)^\frac1m\biggr](u-k)_
  -\partial_t\eta^2\,dtdx. 
 \end{split}
 \]
 Either if $\mu^-\leq\frac12\mu^+$, then
  $\mu^+\leq\omega+\mu^-$ and hence $\mu^+\leq2\omega$. Therefore
 \[
 (\mu^+)^{1-\frac1m}
 \biggl[\biggl(\mu^-+\frac\omega2\biggr)^\frac1m-(\mu^-)^\frac1m\biggr]
 \leq (2\omega)^{1-\frac1m}
 \biggl(\frac\omega2\biggr)^\frac1m
 \leq 2^{1-\frac2m}\omega 
 \]
 and hence
 \[
  I_1\leq C(m)\omega\iint_{Q_{\rho,M}}(u-k)_-\partial_t\eta^2\,dtdx.
 \]
 Otherwise, if $\mu^->\frac12\mu^+$, then
 \[
 \begin{split}
  \Bigl(\mu^-+\frac\omega2\Bigr)^\frac1m-(\mu^-)^\frac1m
  &=\int_0^1\frac{d}{ds}\Bigl(\mu^-+\frac{\omega}{2}s\Bigr)^\frac1m\,ds \\
  &=\frac\omega{2m}\int_0^1 \Bigl(\mu^-+\frac{\omega}{2}s\Bigr)^{\frac1m-1}\,ds 
  \leq\frac\omega{2m}(\mu^-)^{\frac1m-1}
  \leq\frac\omega{2m}\biggl(\frac12\mu^+\biggr)^{\frac1m-1},
 \end{split} 
 \]
 and hence
 \[
 (\mu^+)^{1-\frac1m}
 \biggl[\biggl(\mu^-+\frac\omega2\biggr)^\frac1m-(\mu^-)^\frac1m\biggr]
 \leq \frac\omega{2m}(\mu^+)^{1-\frac1m}\biggl(\frac12\mu^+\biggr)^{\frac1m-1}
 \leq C(m)\omega. 
 \]
 In either case, we obtain
 \begin{equation}
  \label{eq:15}
   I_1\leq C(m)\omega\iint_{Q_{\rho,M}}(u-k)_-\partial_t\eta^2\,dtdx.
 \end{equation}
 Substituting \eqref{eq:14} and \eqref{eq:15} for \eqref{eq:13} we obtain
 \eqref{eq:11}.
\end{proof}

\begin{proof}%
 [Proof of Proposition \ref{prop:1}]
 We consider the scale transform
 \begin{equation*}
  \begin{aligned}
   s&=M^{1-\frac1m}t,&\ \,  
   \tilde{u}(s,x)&=u(t,x),&\ \,
   \tilde{\eta}(s,x)&=\eta(t,x),\\
   \tilde{{f}}(s,x)&={f}(t,x),&\ \,
   g(s,x)&=g(t,x).&& 
  \end{aligned} 
\end{equation*}
 and we put
 $\tilde{h}(\rho,\omega):=\big\|\tilde{f}\big\|^2_{L^q(L^p_\mathrm{w})(Q_{\rho})}
 +\omega\|\tilde{g}\|_{L^\frac{q}2(L^{\frac{p}2}_\mathrm{w})(Q_{\rho})}$. 
 We rewrite the Caccioppoli estimate \eqref{eq:11} as follows:
 \begin{multline}
  \label{eq:16}
  \sup_{s\in I_\rho}\int_{B_\rho}(\tilde{u}(s)-k)_-^2\tilde{\eta}^2(s)\,dx
  +\frac{(\mu^+)^{1-\frac1m}}{M^{1-\frac1m}}
  \iint_{Q_\rho}|\nabla(\tilde{u}-k)_-|^2\tilde{\eta}^2\,dsdx \\
  \leq C(m)\Biggl\{
  \omega\iint_{Q_\rho}(\tilde{u}-k)_-\partial_s\tilde{\eta}^2\,dsdx
  +\frac{(\mu^+)^{1-\frac1m}}{M^{1-\frac1m}}
  \iint_{Q_\rho}(\tilde{u}-k)_-^2|\nabla\tilde{\eta}|^2\,dsdx \\
  +\frac{(\mu^+)^{1-\frac1m}}{M^{1-\frac1m}}\tilde{h}(\rho,\omega)
  \biggl(\int_{I_{\rho}}m_n\bigl(B_\rho\cap\{\tilde{u}(s)<k\}
  \bigr)^{q'(\frac12-\frac1p)}\,ds
  \biggr)^{\frac2{q'}}
  \Biggr\}.
 \end{multline}
 We take $p_*,q_*>0$ as
 \[
  \frac2{q'}=\frac2{q_*}\biggl(1+\frac{2\sigma_0}{n}\biggr),\quad
 q'\biggl(\frac12-\frac1p\biggr)=\frac{q_*}{p_*}.
 \]
 We remark that $\frac2{q_*}+\frac{n}{p_*}=\frac{n}2$. 
 For $i\in\N$, we
 take $\rho=\rho_i,\,k=k_i,\,\tilde{\eta}=\tilde{\eta_i}$ satisfying
 $\tilde{\eta_i}\equiv1$ on $Q_{\rho_{i+1}}$ and
 \[
 \begin{aligned}
  k_i&=\mu^-+\frac14\omega+\frac{1}{2^{i+1}}\omega,\qquad
  \rho_i=\frac12\rho+\frac1{2^{i+1}}\rho,  \\
  Y_i&:=\frac{m_{n+1}\bigl(Q_{\rho_i}\cap\{\tilde{u}<k_i\}\bigr)}%
  {m_{n+1}\bigl(Q_\rho\bigr)}, \\
  Z_i&=\frac{\rho^2}{m_{n+1}\bigl(Q_\rho\bigr)}
  \biggl(
  \int_{I_{\rho_i}}
  m_n\Bigl(B_{\rho_i}\cap\{\tilde{u}(s)<k_i\}\Bigr)^\frac{q_*}{p_*}\,ds
  \biggr)^\frac2{q_*}, \\
  |\nabla\tilde{\eta_i}|&\leq\frac{2}{\rho_i-\rho_{i+1}}
  \leq\frac{8\cdot2^i}{\rho},\qquad
  \partial_s\tilde{\eta_i}\leq\frac{2}{\rho_i^2-\rho_{i+1}^2}\leq\frac{16\cdot2^{2i}}{3\rho^2}.
 \end{aligned}
 \]
 Then, by using \eqref{eq:4} and
 $(\tilde{u}-k_i)_-\leq\frac\omega2$, we rewrite \eqref{eq:16}
 as
 \begin{equation*}
  \begin{split}
   &\quad\|(\tilde{u}-k_i)_-\tilde{\eta_i} 
   \|^2_{L^\infty(L^2)\cap L^2(H^1)(Q_{\rho_i})} \\
   &\leq C(m)\Biggl\{
   \omega\iint_{Q_{\rho_i}}(\tilde{u}-k_i)_-\partial_s\tilde{\eta_i}^2\,dsdx
   +\iint_{Q_{\rho_i}}(\tilde{u}-k_i)_-^2|\nabla\tilde{\eta_i}|^2\,dsdx \\
   &\quad+\tilde{h}(\rho,\omega)\biggl(\int_{I_{\rho_i}}
   m_n\Bigl(B_{\rho_i}\cap\{\tilde{u}(s)<k_i\}
   \Bigr)^\frac{q_*}{p_*}\,ds\biggr)^{\frac2{q_*}(1+\frac{2\sigma_0}{n})}
   \Biggr\} \\
   &\leq C(m)\Biggl\{
  \frac{2^{2i}\omega^2}{\rho^2}m_{n+1}\Bigl(Q_{\rho_i}\cap\{u<k_i\}\Bigr) \\
   &\quad +\tilde{h}(\rho,\omega)
   \biggl(\int_{I_{\rho_i}}m_{n+1}\Bigl(B_{\rho_i}\cap\{\tilde{u}(s)<k_i\}
   \Bigr)^\frac{q_*}{p_*}\,ds\biggr)^{\frac2{q_*}(1+\frac{2\sigma_0}{n})}
   \Biggr\} \\
   &\leq C(m)\frac{\omega^2m_{n+1}\bigl(Q_\rho\bigr)}{\rho^2}\Biggl\{
   2^{2i}Y_i
   +\tilde{h}(\rho,\omega)\omega^{-2}\biggl(\frac{m_{n+1}\bigl(Q_\rho\bigr)}{\rho^2}\biggr)^\frac{2\sigma_0}{n}
   Z_i^{1+\frac{2\sigma_0}{n}}
   \Biggr\}. 
  \end{split} 
\end{equation*}
 Using the Lady\v{z}enskaja inequality (cf. Proposition \ref{prop:4}) and the
 H\"older inequality, we have
 \[
 \begin{split}
  \|(\tilde{u}-k_i)_-\tilde{\eta_i}\|^2_{L^2(Q_{\rho_i})}
  &\leq \|(\tilde{u}-k_i)_-\tilde{\eta_i}\|^2_{L^{2+\frac4n}(Q_{\rho_i})}
  \|\chi_{\{\tilde{u}<k_i\}}\|^2_{L^{n+2}(Q_{\rho_i})} \\
  &\leq C({m,n})\omega^2m_{n+1}\bigl(Q_\rho\bigr)Y_i^{\frac2{n+2}} \\
  &\quad\times\Biggl\{
  2^{2i}Y_i
  +\tilde{h}(\rho,\omega)\omega^{-2}
  \biggl(\frac{m_{n+1}\bigl(Q_\rho\bigr)}{\rho^2}\biggr)^\frac{2\sigma_0}{n}
  Z_i^{1+\frac{2\sigma_0}{n}}
  \Biggr\}
 \end{split}
 \]
 and
\begin{equation*}
   \|(\tilde{u}-k_i)_-\tilde{\eta_i}\|^2_{L^{q_*}(L^{p_*})(Q_{\rho_i})}
  \leq C({m,n})\frac{\omega^2m_{n+1}\bigl(Q_\rho\bigr)}{\rho^2} 
  \Biggl\{
  2^{2i}Y_i
  +\tilde{h}(\rho,\omega)\omega^{-2}
  \biggl(\frac{m_{n+1}\bigl(Q_\rho\bigr)}{\rho^2}\biggr)^\frac{2\sigma_0}{n}
  Z_i^{1+\frac{2\sigma_0}{n}}
  \Biggr\}.
\end{equation*} 
 Since
 \[
 \begin{split}
  \|(\tilde{u}-k_i)_-\tilde{\eta_i}\|^2_{L^2(Q_{\rho_i})}
  &\geq\|(\tilde{u}-k_i)_-\|^2_{L^2(Q_{\rho_{i+1}}\cap\{\tilde{u}<k_{i+1}\})} \\
  &\geq (k_i-k_{i+1})_-^2
  m_{n+1}\Bigl(Q_{\rho_{i+1}}\cap\{\tilde{u}<k_{i+1}\}\Bigr) \\
  &=\frac{\omega^2}{64\cdot2^{2i}}m_{n+1}\bigl(Q_\rho\bigr)Y_{i+1}
 \end{split} 
 \]
and
 \[
 \begin{split}
  \|(\tilde{u}-k_i)_-\tilde{\eta_i}\|^2_{L^{q_*}(L^{p_*})(Q_{\rho_i})} 
  &\geq\|(\tilde{u}-k_i)_-
  \|^2_{L^{q_*}(L^{p_*})(Q_{\rho_{i+1}}\cap\{\tilde{u}<k_{i+1}\})} \\
  &\geq (k_i-k_{i+1})_-^2
  \biggl(\int_{I_{\rho_{i+1}}}
  m_n\Bigl(B_{\rho_{i+1}}\cap\{\tilde{u}(s)<k_{i+1}\}
  \Bigr)^\frac{q_*}{p_*}\,ds\biggr)^{\frac2{q_*}} \\
  &=\frac{\omega^2}{64\cdot2^{2i}}
  \frac{m_{n+1}\bigl(Q_\rho\bigr)}{\rho^2}Z_{i+1},
 \end{split} 
 \]
 we obtain
 \[
 Y_{i+1}
 \leq C({m,n})
 \Biggl\{
 2^{4i}Y_i^{1+\frac2{n+2}}
 +2^{2i}\tilde{h}(\rho,\omega)
 \omega^{-2}\biggl(
 \frac{m_{n+1}\bigl(Q_\rho\bigr)}{\rho^2}\biggr)^\frac{2\sigma_0}{n}
 Y_i^{\frac2{n+2}}Z_i^{1+\varepsilon} \Biggr\} 
 \]
 and
 \[
 Z_{i+1}
 \leq C({m,n})
 \Biggl\{
 2^{4i}Y_i
 +2^{2i}\tilde{h}(\rho,\omega)
 \omega^{-2}\biggl(
 \frac{m_{n+1}\bigl(Q_\rho\bigr)}{\rho^2}\biggr)^\frac{2\sigma_0}{n}
 Z_i^{1+\varepsilon}\Biggr\}. 
 \]
 Either if $q\geq p$, then $\frac{q_*}{p_*}\leq1$ and we obtain
 \begin{equation}
  \label{eq:42}
  \begin{split}
   Z_0&=\frac{\rho^2}{m_{n+1}\bigl(Q_\rho\bigr)}
   \Biggl(\int_{I_{\rho_0}}
   m_n\Bigl(B_{\rho_0}\cap\{\tilde{u}(s)<k_0\}\Bigr)^\frac{q_*}{p_*}
   \,ds\Biggr)^\frac2{q_*} \\
   &\leq\frac{\rho^2}{m_{n+1}\bigl(Q_\rho\bigr)}
   \Biggl(\int_{I_{\rho_0}}m_n\Bigl(B_{\rho_0}\cap\{\tilde{u}(s)<k_0\}\Bigr)
   \,ds\Biggr)^\frac2{p_*}\rho^{\frac4{q_*}(1-\frac{q_*}{p_*})} 
   \leq C({n,p,q})Y_0^\frac2{p_*},
  \end{split}
 \end{equation}
 by the H\"older inequality.
 Otherwise, if $q<p$, then
 \begin{equation}
  \label{eq:43}
  \begin{split}
   Z_0&=\frac{\rho^2}{m_{n+1}\bigl(Q_\rho\bigr)} 
   \Biggl(\int_{I_{\rho_0}}
   m_n\bigl(B_{\rho_0}\cap\{\tilde{u}(s)<k_0\}\bigr)
   m_n\Bigl(B_{\rho_0}\cap\{\tilde{u}(s)<k_0\}\Bigr)^{1-\frac{q_*}{p_*}}
   \,ds\Biggr)^\frac2{q_*} \\
   &\leq\frac{\rho^2}{m_{n+1}\bigl(Q_\rho\bigr)}
   m_n\bigl(B_{\rho_0}\bigr)^{\frac2{p_*}-\frac2{q_*}}
   \Biggl(\int_{I_{\rho_0}}
   m_n\Bigl(B_{\rho_0}\cap\{\tilde{u}(s)<k_0\}\Bigr)
   \,ds\Biggr)^\frac2{q_*} \\
   &\leq C({n,p,q})Y_0^\frac2{q_*}.
  \end{split}
 \end{equation}

 Therefore, by using $\rho^{\sigma_0}\leq\omega
 \tilde{h}(\rho,\omega)^{-\frac12}$ and Lemma \ref{lem:9},
 there exists
 $0<\theta_0=\theta_0(n,m,p,q)<1$ such that if
 $Y_0\leq \theta_0$, then
 $Y_i\rightarrow0$ as $i\rightarrow\infty$, i.e.
\[
 \tilde{u}(s,x)>\mu^-+\frac\omega4
 \quad\text{a.a.}\ (s,x)\in Q_{\frac\rho2}.
\]
\end{proof}



\section{Proof of Upper bounds (\ref{item:2}) of Lemma \ref{lem:1}}
\label{sec:4}
In this section, we prove Upper bounds in Lemma \ref{lem:1}. More
precisely we show the following proposition:

\begin{proposition}
 \label{prop:2} Let $0<\theta_0<1$. Assume inequalities \eqref{eq:4}
 and \eqref{eq:5}. Then, there exist $\eta_1,\delta_1>0$ depending only
 on $n,m,p,q$ and $\theta_0$ such that if
 \[
 \rho^{\sigma_0}\leq\delta_1\omega M^{-\frac1q(1-\frac1m)}
 h(\rho,M,\omega)^{-\frac12}  
 \]
 and
 \[
 m_{n+1}\left(Q_{\rho,M}\cap
 \Bigl\{u<\inf_{Q_{\rho,M}}u+\frac\omega2\Bigr\}\right)
 >\theta_0m_{n+1}\bigl(Q_{\rho,M}\bigr),
 \]
 then
 \[
 u(t,x)\leq\sup_{Q_{\rho,M}}u-\eta_1\omega\quad\text{for}\ 
 (t,x)\in Q^{\theta_0}_{\frac\rho2,M}.
 \]
\end{proposition}
Taking $\theta_0$ as in Proposition \ref{prop:1},
$\delta_1,\,\eta_1>0$ as in Proposition \ref{prop:2} and
\[
 \delta_0=\min\{1,\delta_1\}\,,\,\eta_0=\min\Bigl\{\frac14,\eta_1\Bigr\},
\]
we obtain Lemma \ref{lem:1}.

To prove Proposition \ref{prop:2}, we first show measure estimates of
sub level sets of some time slice.
\begin{lemma}
 \label{lem:4}
 Let $0<\theta_0<1$. If 
\begin{equation}
 \label{eq:17}
  m_{n+1}\left(Q_{\rho,M}\cap\biggl\{u<\mu^-+\frac\omega2\biggr\}\right)
>\theta_0m_{n+1}\bigl(Q_{\rho,M}\bigr),
\end{equation}
 then for all $0<\theta<\theta_0$, there exists 
 $-\frac{\rho^2}{M^{1-\frac1m}}<\tau_0<-\theta\frac{\rho^2}{M^{1-\frac1m}}$
 depending only on $\theta$ and $\theta_0$ such that
 \[
 m_n\biggl(B_\rho\cap\biggl\{u(\tau_0)>\mu^-+\frac{\omega}2\biggr\}\biggr)
 \leq\frac{1-\theta_0}{1-\theta}m_n\bigl(B_\rho\bigr).
 \]
\end{lemma}

\begin{proof}%
 By the change of variable $t=\frac{\rho^2}{M^{1-\frac1m}}s$,
 $\tilde{u}(s,x)=u(t,x)$ and \eqref{eq:17}, we obtain
 \[
 \begin{split}
  \int_{-1}^0
  m_n\left(B_\rho\cap\biggl\{\tilde{u}(s)>\mu^-+\frac\omega2\biggr\}\right)
  \,ds 
  &=\frac{M^{1-\frac1m}}{\rho^2}
  m_{n+1}\left(Q_{\rho,M}\cap\biggl\{u>\mu^-+\frac\omega2\biggr\}\right) \\
  &\leq\frac{M^{1-\frac1m}}{\rho^2}
  \biggl(m_{n+1}\bigl(Q_{\rho,M}\bigr)
  -m_{n+1}\left(Q_{\rho,M}\cap\biggl\{u<\mu^-+\frac\omega2\biggr\}\right)
  \biggr) \\
  &<\frac{M^{1-\frac1m}}{\rho^2}
  (1-\theta_0)m_{n+1}\bigl(Q_{\rho,M}\bigr)
  =(1-\theta_0)m_{n}\bigl(B_\rho\bigr).
 \end{split} 
\]
 If $m_n\left(B_\rho\cap\{u(s)>\mu^-+\frac{\omega}2\}\right)
 >\frac{1-\theta_0}{1-\theta}m_n\left(B_\rho\right)$ for all
 $-1<s<-\theta$, then
 \[
  \begin{split}
   \int_{-1}^0
   m_n\left(B_\rho\cap\biggl\{\tilde{u}(s)>\mu^-+\frac\omega2\biggr\}\right)
   \,ds 
   &\geq \int_{-1}^{\theta_0}
   m_n\left(B_\rho\cap\biggl\{\tilde{u}(s)>\mu^-+\frac\omega2\biggr\}\right)
   \,ds \\
   &\geq(1-\theta_0)m_n\bigl(B_\rho\bigr),
  \end{split}
 \]
 which is contradiction.
\end{proof}

We next show Bernstein type estimates for the positive part of solutions.

\begin{lemma}
 \label{lem:5}
 There exist $r_0,\delta_2>0$ depending only on $n,m,p,q$ and $\theta_0$ such that 
 \[
 m_n\biggl(B_\rho\cap
 \biggl\{u(t)>\mu^+-\frac{\omega}{2^{r_0}}\biggr\}
 \biggr)
 \leq\biggl(1-\biggl(\frac{\theta_0}{2}\biggr)^2\biggr)
 m_n\bigl(B_\rho\bigr)
 \]
 for $t\in I^{\theta_0}_{\rho,M}$, provided
 $\rho^{\sigma_0}
 \leq\delta_2\omega{M^{-\frac1q(1-\frac1m)}}h(\rho,M,\omega)^{-\frac12}$.
\end{lemma}

\begin{proof}%
 We rewrite \eqref{eq:3} as
 \[
  \partial_tu-m u^{1-\frac1m}\Delta u
 =-m u^{1-\frac1m}\Div{f}
 +m u^{1-\frac1m}g.
 \]
 Let 
 \[
 \psi(\xi):=\log_+\Biggl(\frac{H}{H-(\xi-k)_++c}\Biggr),
 \]
 where
 $k=\mu^-+\frac\omega2$\,,\,$
 H=\mu^+-k=\osc_{Q_{\rho,M}}u-\frac\omega2$\,,\,
 $c=\frac{\omega}{2^{r_0}}$
 and $r_0>2$ be chosen later. We remark that $\psi,\psi',\psi''=(\psi')^2\geq0$,
 where $f'=\frac{df}{d\xi}$. We take the cut-off function
 $\eta=\eta(x)$ as 
 \[
 \eta\in C_0^\infty(B_\rho),\ \eta\equiv1\ \text{on}\
 B_{(1-\sigma)\rho}\quad \text{and}\ \,
 |\nabla\eta|\leq\frac{2}{\sigma\rho},
 \]
 where $\sigma>0$ will be chosen later. 
 Putting $w=\psi(u)$ and taking the test function
 $(\psi^2)'(u)\eta^2$ in $(\tau_0,t)\times B_\rho$, 
 where $\tau_0$ will be chosen later, we have
\begin{multline*}
  \frac12\int_{B_\rho}w^2\eta^2\,dx\bigg|_{\tau_0}^t
 +m\int_{\tau_0}^t\int_{B_\rho}
 (\nabla u\cdot\nabla(u^{1-\frac1m}(\psi^2)'\eta^2))\,dtdx \\
 =m\int_{\tau_0}^t\int_{B_\rho}
 ({f}\cdot\nabla((u^{1-\frac1m}(\psi^2)'\eta^2))\,dtdx
 +m\int_{\tau_0}^t\int_{B_\rho}
 u^{1-\frac1m}g(\psi^2)'\eta^2\,dtdx.
\end{multline*} 
Since
 \begin{equation*}
   \nabla((u^{1-\frac1m}(\psi^2)'\eta^2) 
  =\Bigl(1-\frac1m\Bigr)u^{-\frac1m}(\psi^2)'\eta^2\nabla u
  +u^{1-\frac1m}(\psi^2)''\eta^2\nabla u
  +u^{1-\frac1m}(\psi^2)'\nabla\eta^2,
 \end{equation*}   
  we obtain
\begin{equation*}
 \begin{split}
  &\quad\frac12\int_{B_\rho}w^2\eta^2\,dx\bigg|_{\tau_0}^t
  +(m-1)\int_{\tau_0}^t\int_{B_\rho}
  u^{-\frac1m}(\psi^2)'|\nabla u|^2\eta^2\,dtdx 
  +m\int_{\tau_0}^t\int_{B_\rho}
  u^{1-\frac1m}(\psi^2)''|\nabla u|^2\eta^2\,dtdx \\
  &=-m\int_{\tau_0}^t\int_{B_\rho}
  u^{1-\frac1m}(\psi^2)'(\nabla u\cdot\nabla\eta^2)\,dtdx 
  +(m-1)\int_{\tau_0}^t\int_{B_\rho}u^{-\frac1m}(g^2)'
  ({f}\cdot\nabla u)\eta^2\,dtdx \\
  &\quad+m\int_{\tau_0}^t\int_{B_\rho}u^{1-\frac1m}(\psi^2)''
  ({f}\cdot\nabla u)\eta^2\,dtdx 
  +m\int_{\tau_0}^t\int_{B_\rho}u^{1-\frac1m}(\psi^2)'
  ({f}\cdot\nabla\eta^2)\,dtdx \\
  &\quad+m\int_{\tau_0}^t\int_{B_\rho}u^{1-\frac1m}(\psi^2)'
  g\eta^2\,dtdx \\
  &=:I_1+I_2+I_3+I_4+I_5.
 \end{split}
\end{equation*}
 Using the property $(\psi^2)'\nabla u=2w\nabla w$ and the Young inequality, we have
 \[
 \begin{split}
  I_1
  &\leq
  m\int_{\tau_0}^t\int_{B_\rho}u^{1-\frac1m}w|\nabla w|^2\eta^2\,dtdx 
  +4m\int_{\tau_0}^t\int_{B_\rho}u^{1-\frac1m}w|\nabla\eta|^2\,dtdx, \\
  I_2
  &\leq \frac{m-1}2\int_{\tau_0}^t\int_{B_\rho}u^{-\frac1m}(\psi^2)'
  |\nabla u|^2\eta^2\,dtdx 
  +\frac{m-1}2\int_{\tau_0}^t\int_{B_\rho}u^{-\frac1m}(\psi^2)'
  |{f}|^2\eta^2\,dtdx, \\
  I_3
  &\leq \frac{m}{4}\int_{\tau_0}^t\int_{B_\rho}u^{1-\frac1m}(\psi^2)''
  |\nabla u|^2\eta^2\,dtdx
  +m\int_{\tau_0}^t\int_{B_\rho}u^{1-\frac1m}(\psi^2)''
  |{f}|^2\eta^2\,dtdx, \\
  I_4
  &\leq4m\int_{\tau_0}^t\int_{B_\rho}u^{1-\frac1m}w\psi'|{f}||\nabla\eta|\eta\,dtdx \\
  &\leq2m\int_{\tau_0}^t\int_{B_\rho}u^{1-\frac1m}w|\nabla\eta|^2\,dtdx
  +2m\int_{\tau_0}^t\int_{B_\rho}u^{1-\frac1m}(g')^2w|{f}|^2\eta^2\,dtdx, \\
  I_5&
  \leq 2m\int_{\tau_0}^t\int_{B_\rho}
  u^{1-\frac1m}w\psi'|g|\eta^2\,dtdx.
 \end{split}
 \]
 Since $\psi''=(\psi')^2,\,(\psi^2)''=2(\psi')^2(1+\psi)$,
 we have
\begin{multline*}
 m\int_{\tau_0}^t\int_{B_\rho}
 u^{1-\frac1m}(\psi^2)''|\nabla u|^2\eta^2\,dtdx \\
 =2m\int_{\tau_0}^t\int_{B_\rho}
 u^{1-\frac1m}|\nabla w|^2\eta^2\,dtdx
 +2m\int_{\tau_0}^t\int_{B_\rho}
 u^{1-\frac1m}w|\nabla w|^2\eta^2\,dtdx.
\end{multline*} 
 Combining the above estimates, we have
 \begin{equation}
  \label{eq:19}
   \begin{split}
    &\quad\frac12\int_{B_\rho}w^2(t)\eta^2(t)\,dx
    +\frac{m-1}2\int_{\tau_0}^t\int_{B_\rho}
    u^{-\frac1m}(\psi^2)'\eta^2|\nabla u|^2\,dtdx \\
    &\qquad+\frac32m\int_{\tau_0}^t\int_{B_\rho}
    u^{1-\frac1m}|\nabla w|^2\eta^2\,dtdx
    +\frac{m}{2}\int_{\tau_0}^t\int_{B_\rho}
    u^{1-\frac1m}w|\nabla w|^2\eta^2\,dtdx \\
    &\leq
    \frac12\int_{B_\rho}w^2(\tau_0)\eta^2(\tau_0)\,dx
    +6m\int_{\tau_0}^t\int_{B_\rho}
    u^{1-\frac1m}w|\nabla\eta|^2\,dtdx \\
    &\quad+\frac{m-1}2\int_{\tau_0}^t\int_{B_\rho}
    u^{-\frac1m}(\psi^2)'
    |{f}|^2\eta^2\,dtdx 
    +2m\int_{\tau_0}^t\int_{B_\rho}u^{1-\frac1m}(\psi')^2(1+2w)
    |{f}|^2\eta^2\,dtdx \\
    &\quad+2m\int_{\tau_0}^t\int_{B_\rho}u^{1-\frac1m}w\psi'
    |g|\eta^2\,dtdx \\
    &=:I_6+I_7+I_8+I_9+I_{10}.
   \end{split} 
 \end{equation}
 For simplicity, we put
 $k'=\mu^+-c=\mu^+-\frac{\omega}{2^{r_0}}$. First, we estimate the 
 left-hand side of \eqref{eq:19}. Since $k'>k$, we have
 \begin{equation}
  \label{eq:20}
  \begin{split}
   \frac12\int_{B_\rho}w^2(t)\eta^2(t)\,dx
   &\geq\frac12\int_{B_{(1-\sigma)\rho}\cap\{u(t)>k'\}}w^2(t)\,dx \\
   &\geq\frac12\int_{B_{(1-\sigma)\rho}\cap\{u(t)>k'\}}
   \log^2\biggl(\frac{H}{H-(k'-k)+c}\biggr)\,dx \\
   &\geq\frac12\log^2\biggl(\frac{\frac\omega4}{\frac\omega{2^{r_0-1}}}\biggr)
   m_n\Bigl(B_{(1-\sigma)\rho}\cap\{u(t)>k'\}\Bigr) \\
   &=\frac12(r_0-3)^2\log^22m_n\Bigl(B_{(1-\sigma)\rho}\cap\{u(t)>k'\}\Bigr).
  \end{split}
 \end{equation}
 Second, we estimate $I_6$. Taking $\tau_0$ as in Lemma \ref{lem:4}
 with $\theta=\frac{\theta_0}{2}$, we obtain
 \[
 w=\log_+\biggl(\frac{H}{H-(u-k)_++c}\biggr)
 \leq\log\biggl(\frac{\frac12\omega}{\frac{1}{2^{r_0}}\omega}\biggr)
 =(r_0-1)\log2
 \]
 and hence
 \begin{equation}
   \begin{split}
    I_6&\leq\frac12\int_{B_\rho\cap\{u({\tau_0})>k\}}w^2({\tau_0})\,dx \\
    &\leq\frac12(r_0-1)^2\log^22|B_\rho\cap\{u({\tau_0})>k\}| 
    \leq\frac12\cdot\frac{1-\theta_0}{1-\frac{\theta_0}2}
   (r_0-1)^2\log^22m_n\bigl(B_\rho\bigr).
   \end{split}
 \end{equation}
 We estimate $I_7$. From
 $t-{\tau_0}\leq\frac{\rho^2}{M^{1-\frac1m}}$ and
 \eqref{eq:4}, we have
 \begin{equation}
  \begin{split}
   I_7
   &\leq6m(\mu^+)^{1-\frac1m}(t-{\tau_0})(r_0-1)\log2
   \biggl(\frac{2}{\sigma\rho}\biggr)^2m_n\bigl(B_\rho\bigr) \\
   &\leq C(m)\Bigl(\frac{r_0-1}{\sigma^2}\Bigr)m_n\bigl(B_\rho\bigr).
  \end{split} 
\end{equation}
 We estimate $I_8$. Since
 \[
 \psi'\leq\frac{1}{H-(u-k)_++c}
 \leq\frac1c=\frac{2^{r_0}}{\omega},\quad 
 (\psi^2)'=2\psi\psi'
 \leq\frac{2^{r_0+1}}{\omega}(r_0-1)\log2
 \]
 and 
 \[
 u^{-\frac1m}\leq k^{-\frac1m}
 \leq\biggl(\frac\omega2\biggr)^{-\frac1m}
 \quad \text{for}\ u\geq k, 
 \]
 we have
 \[
 I_8
 \leq C(m)(r_0-1)2^{r_0}
 \omega^{-1-\frac1m}
 \int_{\tau_0}^t\int_{B_\rho\cap\{u(s)>k\}}|{f}|^2\,dtdx.
 \]
 By the definition of the weak $L^p$ space and by the H\"older inequality, we
 have
 \[
 \begin{split}
  \int_{\tau_0}^t\int_{B_\rho\cap\{u(s)>k\}}|{f}|^2\,dtdx 
  &\leq\int_{\tau_0}^t
  \big\||{f}(s)|^2\big\|_{L^\frac{p}{2}_\mathrm{w}(B_\rho)}
  m_n\Bigl(B_\rho\cap\{u(s)>k\}\Bigr)^{1-\frac2p}\,ds \\ 
  &\leq C(n,p)
  M^{\frac2q(1-\frac1m)}
  \big\||{f}|^2\big\|_{L^\frac{q}{2}(L^\frac{p}{2}_\mathrm{w})(Q_{\rho,M})}
  \rho^{2\sigma_0}
 \frac{m_n\bigl(B_\rho\bigr)}{M^{1-\frac1m}}.
 \end{split} 
\]
Using \eqref{eq:5}, we obtain
 \begin{equation}
   \begin{split}
    I_8
    &\leq C(n,m,p)
    \Bigl(\frac{\rho^{2\sigma_0}}{\omega^2}
    M^{\frac2q(1-\frac1m)}
    \big\||{f}|^2\big\|_{L^\frac{q}{2}(L^\frac{p}{2}_\mathrm{w})(Q_{\rho,M})}
    \Bigr) \\
    &\quad\times
    \Bigl(\frac\omega{M}\Bigr)^{1-\frac1m}
    (r_0-1)2^{r_0}      
    m_n\bigl(B_\rho\bigr) \\
    &\leq C(n,m,p)
    \Bigl(\frac{\rho^{2\sigma_0}}{\omega^2}
    M^{\frac2q(1-\frac1m)}
    \big\||{f}|^2\big\|_{L^\frac{q}{2}(L^\frac{p}{2}_\mathrm{w})(Q_{\rho,M})}
    \Bigr)
    (r_0-1)2^{r_0}      
    m_n\bigl(B_\rho\bigr).
   \end{split} 
 \end{equation}
 We estimate $I_9$ and $I_{10}$. Considering the same calculation for
 $I_8$, we have
 \begin{equation}
    I_9
   \leq C(n,m,p)
   \Bigl(\frac{\rho^{2\sigma_0}}{\omega^2}
   M^{\frac2q(1-\frac1m)}
   \big\||{f}|^2\big\|_{L^\frac{q}{2}(L^\frac{p}{2}_\mathrm{w})(Q_{\rho,M})}
   \Bigr) 
   2^{2r_0}(1+2(r_0-1)\log2)
 m_n\bigl(B_\rho\bigr)
 \end{equation}
 and
 \begin{equation}
  \label{eq:37}
   I_{10}
   \leq C(n,m,p)
   \Bigl(\frac{\rho^{2\sigma_0}}{\omega}
   M^{\frac2q(1-\frac1m)}
   \big\|g\big\|_{L^\frac{q}{2}(L^\frac{p}{2}_\mathrm{w})(Q_{\rho,M})}
   \Bigr)
   2^{r_0}(r_0-1)
   m_n\bigl(B_\rho\bigr).
 \end{equation}
 Combining estimates
 \eqref{eq:20}--\eqref{eq:37},
 we have
 \begin{multline*}
 m_n\Bigl(B_{(1-\sigma)\rho}\cap\{u(t)>k'\}\Bigr)
  \leq\Biggl\{\frac{1-\theta_0}{1-\frac{\theta_0}2}
  \biggl(\frac{r_0-1}{r_0-3}\biggr)^2
  +\frac{C_1(m)}{\sigma^2}\frac{r_0-1}{(r_0-3)^2} \\
  +C_2(n,m,p)
  \Bigl(\frac{\rho^{2\sigma_0}}{\omega^2}
   M^{\frac2q(1-\frac1m)}
  \big\||{f}|^2\big\|_{L^\frac{q}{2}(L^\frac{p}{2}_\mathrm{w})(Q_{\rho,M})}
  \Bigr)
  \frac{2^{r_0}(r_0-1)}{(r_0-3)^2} \\
  +C_3(n,m,p)
  \Bigl(\frac{\rho^{2\sigma_0}}{\omega^2}
   M^{\frac2q(1-\frac1m)}
  \big\||{f}|^2\big\|_{L^\frac{q}{2}(L^\frac{p}{2}_\mathrm{w})(Q_{\rho,M})}
  \Bigr)
  \frac{2^{2r_0}(1+2(r_0-1)\log2)}{(r_0-3)^2} \\
  +C_4(n,m,p)
  \Bigl(\frac{\rho^{2\sigma_0}}{\omega}
  M^{\frac2q(1-\frac1m)}
  \big\|g\big\|_{L^\frac{q}{2}(L^\frac{p}{2}_\mathrm{w})(Q_{\rho,M})}
  \Bigr)
  \frac{2^{r_0}(r_0-1)}{(r_0-3)^2}
  \Biggr\}m_n\bigl(B_\rho\bigr).
 \end{multline*}
 Since
 \[
\begin{split}
 m_n\bigl(B_\rho\cap\{u(t)>k'\}\bigr)
 &=m_n\Bigl((B_\rho\setminus B_{(1-\sigma)\rho})\cap\{u(t)>k'\}\Bigr) 
 +m_n\Bigl(B_{(1-\sigma)\rho}\cap\{u(t)>k'\}\Bigr) \\
 &\leq(1-(1-\sigma)^n)m_n\bigl(B_\rho\bigr) 
 +m_n\Bigl(B_{(1-\sigma)\rho}\cap\{u(t)>k'\}\Bigr),
\end{split} 
\]
 we have
 \begin{multline*}
  m_n\Bigl(B_{(1-\sigma)\rho}\cap\{u(t)>k'\}\Bigr) \\
  \leq\Biggl\{\frac{1-\theta_0}{1-\frac{\theta_0}2}
  \biggl(\frac{r_0-1}{r_0-3}\biggr)^2
  +\frac{C_1(m)}{\sigma^2}\frac{r_0-1}{(r_0-3)^2} 
  +(1-(1-\sigma)^n)  \\
  +\max\{C_2,C_3,C_4\}
  \frac{\rho^{2\sigma_0}}{\omega^2}
  M^{\frac2q(1-\frac1m)}h(\rho,M,\omega)
  C_5(r_0)
  \Biggr\}m_n\bigl(B_\rho\bigr),
 \end{multline*}
 where 
 \[
  C_5(r_0)=\max\biggl\{\frac{2^{r_0}(r_0-1)}{(r_0-3)^2},
 \frac{2^{2r_0}(1+2(r_0-1)\log2)}{(r_0-3)^2}\biggr\}.
 \]
 We choose parameters $r_0,\sigma$ and $\delta_2$. First we choose
 $\sigma=\sigma(n,\theta_0)$ satisfying
 $1-(1-\sigma)^n\leq\frac18\theta^2_0$. Second, we choose
 $r_0=r_0(n,m,\theta_0)$ satisfying
 \[
 \biggl(\frac{r_0-1}{r_0-3}\biggr)^2\leq\biggl(1-\frac{\theta_0}2\biggr)(1+\theta_0)\ \,
 \text{and}\ \,
 \frac{C_1(m)}{\sigma^2}\frac{r_0-1}{(r_0-3)^2}\leq\frac18\theta_0^2.
 \]
 Finally, we choose $\delta_2=\delta_2(n,m,p,\theta_0)>0$
 sufficiently small such that
 \[
 \max\{C_2,C_3,C_4\}
 C_5(r_0)
 \delta_2\leq\frac12\theta_0^2.
 \]
 Then, if $\rho^{2\sigma_0}
 \leq\delta_2\omega^2M^{-\frac2q(1-\frac1m)}h(\rho,M,\omega)^{-1}$, we have
 \[
 m_n\Bigl(B_\rho\cap\{u(t)>k'\}\Bigr)
 \leq\biggl(1-\biggl(\frac{\theta_0}{2}\biggr)^2\biggr)
 m_n\bigl(B_\rho\bigr).
 \]
\end{proof}

In the proof of Proposition \ref{prop:2}, we need to show the
Caccioppoli estimate.

\begin{lemma}[the Caccioppoli estimate for super level sets]
\label{lem:6}
 Let $\eta=\eta(t,x)$ be a cut-off function in
 $Q^{\theta_0}_{\rho,M}$. For $k\geq\mu^+-\frac\omega2$, there exists
 a constant $C>0$ depending only on $m$ such that
 \begin{multline}
  \label{eq:18}
  \sup_{t\in I^{\theta_0}_{\rho,M}}
  \int_{B_\rho}(u(t)-k)_+^2\eta^2(t)\,dx
  +M^{1-\frac1m}\iint_{Q^{\theta_0}_{\rho,M}}
  |\nabla (u-k)_+|^2\eta^2\,dtdx \\
  \leq
  C\biggl\{
  \Bigl(\frac{M}{\mu^+}\Bigr)^{1-\frac1m}
  \iint_{Q^{\theta_0}_{\rho,M}}(u-k)_+^2
  \partial_t\eta^2\,dtdx 
  +M^{1-\frac1m}\iint_{Q^{\theta_0}_{\rho,M}}(u-k)_+^2
  |\nabla\eta|^2\,dtdx \\
  +M^{1-\frac1m}h(\rho,M,\omega)
  \biggl(\int_{t\in I^{\theta_0}_{\rho,M}}
  m_n\Bigl(B_\rho\cap\{u(t)>k\}\Bigr)^{q'(\frac12-\frac1p)}\,dt
  \biggr)^{\frac{2}{q'}} \biggr\},
 \end{multline}
 where $\frac12=\frac1q+\frac1{q'}$.
\end{lemma}

\begin{proof}%
 Testing a function $(u-k)_+\eta^2$ to \eqref{eq:3}, we have
 \begin{equation}
  \label{eq:25}
   \begin{split}
    &\quad\frac1m\sup_{t\in I^{\theta_0}_{\rho,M}}
    \int_{B_\rho}
    \biggl(\int_0^{(u(t)-k)_+}(k+\xi)^{\frac1m-1}\xi\,d\xi\biggr)\eta^2(t)\,dx 
    +\iint_{Q^{\theta_0}_{\rho,M}}
    |\nabla(u-k)_+|^2\eta^2\,dtdx \\
    &\leq \frac1m\iint_{Q^{\theta_0}_{\rho,M}}
    \biggl(\int_0^{(u-k)_+}(k+\xi)^{\frac1m-1}\xi\,d\xi\biggr)
    \partial_t\eta^2\,dtdx 
    -\iint_{Q^{\theta_0}_{\rho,M}}
    (\nabla(u-k)_+\cdot\nabla\eta^2)(u-k)_+\,dtdx \\
    &\quad+\iint_{Q^{\theta_0}_{\rho,M}}
    {f}\cdot\nabla(u-k)_+\eta^2\,dtdx
    +\iint_{Q^{\theta_0}_{\rho,M}}
    ({f}\cdot\nabla\eta^2)(u-k)_+\,dtdx \\
    &\quad+\iint_{Q^{\theta_0}_{\rho,M}}
    g(u-k)_+\eta^2\,dtdx \\
    &=:I_1+I_2+I_3+I_4+I_5.
   \end{split} 
 \end{equation}
 By the Young inequality and since $k>\mu^+-\frac\omega2$, we have
\begin{equation}
 \label{eq:26}
 \begin{split}
  I_2&\leq\frac12\iint_{Q^{\theta_0}_{\rho,M}}|\nabla(u-k)_+|^2\eta^2\,dtdx
  +2\iint_{Q^{\theta_0}_{\rho,M}}(u-k)_+^2|\nabla\eta|^2\,dtdx,\\
  I_3&\leq\frac14\iint_{Q^{\theta_0}_{\rho,M}}|\nabla(u-k)_+|^2\eta^2\,dtdx
  +\iint_{Q^{\theta_0}_{\rho,M}\cap\{u>k\}}|{f}|^2\eta^2\,dtdx, \\
  I_4&\leq\iint_{Q^{\theta_0}_{\rho,M}}(u-k)_+^2|\nabla\eta|^2\,dtdx
  +\iint_{Q^{\theta_0}_{\rho,M}\cap\{u>k\}}|{f}|^2\eta^2\,dtdx, \\
  I_5&\leq\frac\omega2
  \iint_{Q^{\theta_0}_{\rho,M}\cap\{u>k\}}|g|\eta^2\,dtdx. \\
 \end{split}
\end{equation}
 We estimate the first term of the left-hand side in
 \eqref{eq:25}. Since
 \[
 (k+\xi)^{\frac1m-1}
 \geq  u^{\frac1m-1}
 \geq (\mu^+)^{\frac1m-1}
 \geq M^{\frac1m-1}\quad\text{for}\quad 0\leq\xi\leq(u-k)_+,
 \]
 we have
 \begin{equation}
  \label{eq:27}
   \int_0^{(u(t)-k)_+}(k+\xi)^{\frac1m-1}\xi\,d\xi
   \geq\frac12M^{\frac1m-1}(u(t)-k)_+^2.
 \end{equation}
 Finally, we estimate $I_1$. By \eqref{eq:5}, 
 we have
 \[
 (k+\xi)^{\frac1m-1}
 \leq k^{\frac1m-1}
 \leq \Bigl(\mu^+-\frac\omega2\Bigr)^{\frac1m-1}
 \leq \Bigl(\frac13\mu^+\Bigr)^{\frac1m-1}
 \]
 and hence
 \begin{equation}
  \label{eq:28}
   I_1\leq\frac12\Bigl(\frac13\mu^+\Bigr)^{\frac1m-1}
   \iint_{Q_{\rho,M}^{\theta_0}}
   (u-k)_+^2\partial_t\eta^2\,dtdx.
 \end{equation}
 Combining estimates
 \eqref{eq:26},~\eqref{eq:27} and \eqref{eq:28}, we obtain
 \begin{multline}
  \label{eq:44}
  M^{\frac1m-1}
  \sup_{t\in I^{\theta_0}_{\rho,M}}
  \int_{B_\rho}
  (u(t)-k)_+^2\eta^2(t)\,dx
  +\iint_{Q^{\theta_0}_{\rho,\omega}}
  |\nabla(u-k)_+|^2\eta^2\,dtdx \\
  \leq C(m)\biggl\{(\mu^+)^{\frac1m-1}
  \iint_{Q^{\theta_0}_{\rho,M}}
  (u-k)_+^2\partial_t\eta^2\,dtdx
  +\iint_{Q^{\theta_0}_{\rho,M}}
  (u-k)_+^2|\nabla\eta|^2\,dtdx \\
  +\iint_{Q^{\theta_0}_{\rho,M}\cap\{u>k\}}
  |{f}|^2\eta^2\,dtdx
  +\frac\omega2
  \iint_{Q^{\theta_0}_{\rho,M}\cap\{u>k\}}|g|\eta^2\,dtdx
  \biggr\}. 
 \end{multline}
 Using the same argument of the proof of Lemma \ref{lem:3}, we have
 \begin{equation*}
  \iint_{Q^{\theta_0}_{\rho,M}\cap\{u>k\}}
   |{f}|^2\eta^2\,dtdx 
  \leq\big\||{f}|^2\big\|_{L^\frac{q}{2} L^\frac{p}{2}_\mathrm{w}(Q^{\theta_0}_{\rho,M})}
  \Biggl(
  \int_{I^{\theta_0}_{\rho,M}}
  m_n\Bigl(B_\rho\cap\{u(t)>k\}\Bigr)^{q'(\frac12-\frac1p)}\,dt
  \Biggr)^{\frac2{q'}}, 
 \end{equation*} 
 and
\begin{equation*}
   \iint_{Q^{\theta_0}_{\rho,M}\cap\{u>k\}}
  |g|\eta^2\,dtdx 
   \leq\|g\|_{L^\frac{q}{2} L^\frac{p}{2}_\mathrm{w}(Q^{\theta_0}_{\rho,M})}
  \Biggl(
  \int_{I^{\theta_0}_{\rho,M}}
  m_n\Bigl(B_\rho\cap\{u(t)>k\}\Bigr)^{q'(\frac12-\frac1p)}\,dt
 \Biggr)^{\frac2{q'}},
\end{equation*} 
 hence we obtain \eqref{eq:18} from \eqref{eq:44}.
\end{proof}

Using Bernstein type estimates, the Caccioppoli estimate and the hole
filling argument, we may prove the smallness of measures of super
level sets.

\begin{lemma}
 \label{lem:7}
 Let $\rho_0=\frac34\rho$. For $0<\nu<1$, there exist
 $q_0\,,\,\delta_1>0$ depending only on
 $n,m,p,q,\theta_0$ and $\nu$ such that
 \[
 m_{n+1}\bigg(
 Q_{\rho_0,M}^{\theta_0}\cap
 \Bigl\{u>\mu^+-\frac{\omega}{2^{q_0+1}}\Bigr\}
 \bigg)
 \leq\nu
 m_{n+1}\bigl(Q_{\rho_0,M}^{\theta_0}\bigr)
 \]
 provided $\rho^{\sigma_0}\leq
 \delta_2\omega{M^{-\frac1q(1-\frac1m)}}h(\rho,M,\omega)$.
\end{lemma}

\begin{remark}
 We obtain the estimate of $\delta_1$ as
 \[
  \delta_1\leq\theta_0^{\frac1q-\frac12}2^{-q_0}.
 \]
\end{remark}

\begin{proof}%
 [Proof of Lemma \ref{lem:7}]
 We fix $t\in I_{\rho,M}^{\theta_0}$ and set
 \[
 l:=\mu^+-\frac{\omega}{2^{j+1}},\quad
 k:=\mu^+-\frac{\omega}{2^j},
 \]
 where $j\geq r_0$ and the constant $r_0$ is given by Lemma
 \ref{lem:5}. By the Poincar\'e type inequality
 (cf. Proposition \ref{prop:3}), we
 have
 \begin{equation*}
  \frac{\omega}{2^{j+1}}m_n\bigl(B_{\rho_0}\cap\{u(t)>l\}\bigr) 
  \leq\frac{C(n)\rho_0^{n+1}}%
   {m_n\bigl(B_{\rho_0}\cap\{u(t)\leq k\}\bigr)}
   \int_{B_{\rho_0}\cap\{k<u(t)\leq l\}}|\nabla u(t)|\,dx.
 \end{equation*}
 Since $k>\mu^+-\frac{\omega}{2^{r_0}}$ and Lemma \ref{lem:5}, we have
 \[
  m_n\bigl(B_{\rho_0}\cap\{u(t)\leq k\}\bigr)
 =m_n\bigl(B_{\rho_0}\bigr)
 -m_n\bigl(B_{\rho_0}\cap\{u(t)>k\}\bigr) 
 \geq\biggl(\frac{\theta_0}{2}\biggr)^2
 m_n\bigl(B_{\rho_0}\bigr)
 \]
 and hence
 \begin{equation}
  \label{eq:29}
   \frac{\omega}{2^{j+1}}m_n\bigl(B_{\rho_0}\cap\{u(t)>l\}\bigr)
   \leq\frac{C(n)\rho_0}{\theta_0^2}
   \int_{B_{\rho_0}\cap\{k<u(t)\leq l\}}|\nabla u(t)|\,dx.
 \end{equation}
 Integrating over
 $I_{\rho_0,M}^{\theta_0}$ for
 \eqref{eq:29}, we obtain
 \[
  \begin{split}
   \frac{\omega}{2^{j+1}}
   m_{n+1}\bigl(Q^{\theta_0}_{\rho_0,M}\cap\{u>l\}\bigr)
   &\leq\frac{C(n)\rho_0}{\theta_0^2}
   \int_{I_{\rho_0,M}^{\theta_0}}
   \int_{B_{\rho_0}\cap\{k<u(t)\leq l\}}|\nabla u(t)|\,dtdx \\
   &\leq\frac{C(n)\rho_0}{\theta_0^2}
   \|\nabla(u-k)_+\|_{L^2(Q^{\theta_0}_{\rho_0,M})} 
   m_{n+1}\bigl(Q^{\theta_0}_{\rho_0,M}\cap\{k<u\leq l\}\bigr)^\frac12.
  \end{split}
 \]
 We estimate $\|\nabla(u-k)_+\|_{L^2(Q^{\theta_0}_{\rho_0,M})}$.
 Let $\eta=\eta(t,x)$ be a cut-off function in $Q^{\theta_0}_{\rho,M}$
 satisfying 
 \[
 \eta\equiv1\ \text{on}\ Q_{\rho_0,M}^{\theta_0},\quad
 |\nabla\eta|\leq\frac{8}{\rho}\quad\text{and}\quad
 \partial_t\eta\leq\frac{10M^{1-\frac1m}}{\theta_0\rho^2}.
 \]
 Then, by
 the Caccioppoli estimate (Lemma \ref{lem:6}), we have
 \begin{equation}
  \label{eq:30}
   \begin{split}
    \|\nabla(u-k)_+\|^2_{L^2(Q^{\theta_0}_{\rho_0,M})} 
    &\leq\|\nabla(u-k)_+\eta\|^2_{L^2(Q^{\theta_0}_{\rho,M})} \\
    &\leq C(m) \Biggl\{
    \iint_{Q^{\theta_0}_{\rho,M}}
    (u-k)_+^2(|\nabla\eta|^2+(\mu^+)^{\frac1m-1}\partial_t\eta^2)
    \,dtdx \\
    &\quad+h(\rho,M,\omega)\biggl(
    \int_{I_{\rho,M}^{\theta_0}}
    m_n\bigl(B_\rho\cap\{u(t)>k\}\bigr)^{q'(\frac12-\frac1p)}\,dt
    \biggr)^\frac{2}{q'}
    \Biggr\} \\
    &=:I_1+I_2.
   \end{split} 
 \end{equation}
 First we estimate $I_1$. By the inequality \eqref{eq:4}, we have
 \begin{equation}
  \label{eq:32} 
  \begin{split}
   I_1
   &\leq C(m)(\mu^+-k)_+^2
   \biggl(\frac1{\rho^2}
   +\frac{M^{1-\frac1m}}{\theta_0\rho^2}(\mu^+)^{\frac1m-1}
   \biggr)
   m_{n+1}\bigl(Q^{\theta_0}_{\rho_0,M}\bigr)  \\
   &\leq C(m)\biggl(\frac\omega{2^j}\biggr)^2
   \frac{1}{\theta_0\rho^2}
   \biggl(\frac{M}{\mu^+}\biggr)^{1-\frac1m}
   m_{n+1}\bigl(Q^{\theta_0}_{\rho_0,M}\bigr) \\
   &\leq C(m)\biggl(\frac\omega{2^j}\biggr)^2
   \frac{1}{\theta_0\rho^2}
   m_{n+1}\bigl(Q^{\theta_0}_{\rho_0,M}\bigr).
  \end{split}  
\end{equation} 

 We estimate $I_2$. Since
 \begin{equation*}
  \begin{split}
   &\quad\biggl(\int_{I_{\rho,M}^{\theta_0}}
   m_n\bigl(B_\rho\cap\{u(t)>k\}\bigr)^{q'(\frac12-\frac1p)}
   \,dt\biggr)^\frac{2}{q'} \\
   &\leq m_n\bigl(B_\rho\bigr)^{1-\frac2p}
   \Bigl(\frac{\theta_0}2\frac{\rho^2}{M^{1-\frac1m}}\Bigr)^\frac{2}{q'} \\
   &\leq C(q)m_n\bigl(B_\rho\bigr)^{-\frac2p}
   \Bigl(\frac{\theta_0}2\frac{\rho^2}{M^{1-\frac1m}}\Bigr)^{\frac{2}{q'}-1} 
   m_{n+1}\bigl(Q^{\theta_0}_{\rho_0,M}\bigr)
  \\
   &\leq C(q)m_n\bigl(B_\rho\bigr)^{-\frac2p}
  \Bigl(\frac{\theta_0}2\frac{\rho^2}{M^{1-\frac1m}}
   \Bigr)^{\frac{2}{q'}-1}
   m_{n+1}\bigl(Q^{\theta_0}_{\rho_0,M}\bigr) \\
  &\leq C(n,p,q)
  \biggl(\rho^{2\sigma_0}M^{\frac2q(1-\frac1m)}
  \biggl(\frac{2^j}{\omega}\biggr)^2\theta_0^\frac2{q'}\biggr)
  \frac1{\theta_0\rho^2}\biggl(\frac{\omega}{2^j}\biggr)^2
   m_{n+1}\bigl(Q^{\theta_0}_{\rho_0,M}\bigr),
  \end{split} 
 \end{equation*} 
 we obtain
\begin{equation}
   \label{eq:31}
  I_2\leq C(n,m,p,q)
  \biggl(\rho^{2\sigma_0}M^{\frac2q(1-\frac1m)}
  \biggl(\frac{2^j}{\omega}\biggr)^2\theta_0^\frac2{q'}h(\rho,M,\omega)
  \biggr) 
  \times\frac1{\theta_0\rho^2}
  \biggl(\frac{\omega}{2^j}\biggr)^2
  m_{n+1}\bigl(Q^{\theta_0}_{\rho_0,M}\bigr).
\end{equation}
 Combining estimates \eqref{eq:30}, \eqref{eq:32} and \eqref{eq:31}, we
 obtain
\begin{multline*}
  \|\nabla(u-k)_+\|^2_{L^2(Q^{\theta_0}_{\rho_0,M})} \\
 \leq C(n,m,p,q)
 \Bigl(1+
 \rho^{2\sigma_0}M^{\frac2q(1-\frac1m)}
 \biggl(\frac{2^j}{\omega}\biggr)^2\theta_0^\frac2{q'}h(\rho,M,\omega)
 \Bigr) 
 \frac{1}{\theta_0\rho^2}
 \biggl(\frac{\omega}{2^j}\biggr)^2
 m_{n+1}\bigl(Q^{\theta_0}_{\rho_0,M}\bigr)
\end{multline*} 
 and hence
 \begin{multline*}
  \biggl(\frac{\omega}{2^{j+1}}\biggr)^2
  m_{n+1}\bigl(Q_{\rho_0,M}^{\theta_0}\cap\{u>l\}\bigr)^2 \\
  \leq\frac{C(n,m,p,q)}{\theta_0^5}\biggl(\frac{\omega}{2^j}\biggr)^2
  \Bigl(1+
  \rho^{2\sigma_0}M^{\frac2q(1-\frac1m)}
  \biggl(\frac{2^j}{\omega}\biggr)^2\theta_0^\frac2{q'}h(\rho,M,\omega)
  \Bigr) \\
  \times
  m_{n+1}\bigl(Q_{\rho_0,M}^{\theta_0}\bigr)\cdot
  m_{n+1}\bigl(Q_{\rho_0,M}^{\theta_0}\cap\{k<u\leq l\}\bigr).
 \end{multline*} 
 Summing over $i=r_0+1,\dots,q_0$, we have
 \begin{equation*}
\begin{split}
 &\quad\sum_{i=r_0+1}^{q_0}
  m_{n+1}\biggl(
  Q_{\rho_0,M}^{\theta_0}
  \cap\Bigl\{u>\mu^+-\frac{\omega}{2^{i+1}}\Bigr\}
  \biggr)^2 \\
  &\leq\frac{C(n,m,p,q)}{\theta_0^5}
  m_{n+1}\bigl(Q_{\rho_0,M}^{\theta_0}\bigr) \\
 &\quad\times\sum_{i=r_0+1}^{q_0}
 \Bigl(1+
 \rho^{2\sigma_0}M^{\frac2q(1-\frac1m)}
 \biggl(\frac{2^j}{\omega}\biggr)^2\theta_0^\frac2{q'}h(\rho,M,\omega)
 \Bigr) 
 m_{n+1}\biggl(Q_{\rho_0,M}^{\theta_0}
  \cap\Bigl\{\mu^+-\frac{\omega}{2^i}<u
  \leq\mu^+-\frac{\omega}{2^{i+1}}\Bigr\}
  \biggr) \\
  &\leq\frac{C(n,m,p,q)}{\theta_0^5}
  m_{n+1}\bigl(Q_{\rho_0,M}^{\theta_0}\bigr) 
  \Bigl(1+
  \rho^{2\sigma_0}M^{\frac2q(1-\frac1m)}
  \biggl(\frac{2^{q_0}}{\omega}\biggr)^2\theta_0^\frac2{q'}h(\rho,M,\omega)
  \Bigr) \\
  &\quad\times\sum_{i=r_0+1}^\infty
  m_{n+1}\biggl(Q_{\rho_0,M}^{\theta_0}
  \cap\Bigl\{\mu^+-\frac{\omega}{2^i}<u 
  \leq\mu^+-\frac{\omega}{2^{i+1}}\Bigr\}
  \biggr) \\
 &\leq \frac{C(n,m,p,q)}{\theta_0^5} 
  m_{n+1}\bigl(Q_{\rho_0,M}^{\theta_0}\bigr)^2 
  \Bigl(1+
  \rho^{2\sigma_0}M^{\frac2q(1-\frac1m)}
  \biggl(\frac{2^{q_0}}{\omega}\biggr)^2\theta_0^\frac2{q'}h(\rho,M,\omega)
  \Bigr).
\end{split} 
 \end{equation*}
 We take $q_0>0$ enough large such that
 \[
 \frac{2C({n,m,p,q})}{\theta_0^5(q_0-r_0)}\leq\nu^2.
 \]
 Since
\begin{equation*}
   \sum_{i=r_0+1}^{q_0}
  m_{n+1}\biggl(Q_{\rho_0,M}^{\theta_0}
  \cap\Bigl\{u>\mu^+-\frac{\omega}{2^{i+1}}\Bigr\}
  \biggr)^2 
  \geq(q_0-r_0)
  m_{n+1}
  \biggl(
  Q_{\rho_0,M}^{\theta_0}\cap\Bigl\{u>\mu^+-\frac{\omega}{2^{q_0+1}}\Bigr\}
 \biggr)^2, 
\end{equation*}
 we have
 \[
 \begin{split}
  m_{n+1}\biggl(Q_{\rho_0,M}^{\theta_0}
  \cap\Bigl\{u>\mu^+-\frac{\omega}{2^{q_0+1}}\Bigr\}
  \biggr)^2 
  &\leq\frac{2C({n,m,p,q})}{\theta_0^5(q_0-r_0)}
  m_{n+1}\bigl(Q_{\rho_0,M}^{\theta_0}\bigr)^2 \\
  &\leq\nu^2
  m_{n+1}\bigl(Q_{\rho_0,M}^{\theta_0}\bigr)^2 
 \end{split} 
\]
 provided $\rho^{2\sigma_0}\leq
 \min\{\theta_0^{-\frac{2}{q'}}2^{-2q_0},\delta_2^2\}\omega^2
 M^{-\frac2q(1-\frac1m)}h(\rho,M,\omega)^{-1}$, where $\delta_2>0$ is given by
 Lemma \ref{lem:5}. Taking
 $\delta_1^2:=\min\{\theta_0^{-\frac{2}{q'}}2^{-2q_0},\delta_2^2\}$, we
 obtain Lemma \ref{lem:7}.
\end{proof}

\begin{proof}%
 [Proof of Proposition \ref{prop:2}] 
 Let $0<\nu<1$ be chosen later. We take $\delta_1>0$ and $q_0$ as in
 Lemma \ref{lem:7}.  We introduce the following scale transform
 \begin{equation*}
  \begin{aligned}
   s&=M^{1-\frac1m}t,&\ \,  
   \tilde{u}(s,x)&=u(t,x),&\ \,
   \tilde{\eta}(s,x)&=\eta(t,x),\\
   \tilde{{f}}(s,x)&={f}(t,x),&\ \,
   \tilde{g}(s,x)&=g(t,x).
  \end{aligned} 
\end{equation*}
Then, using \eqref{eq:4}, we may rewrite the Caccioppoli estimate
\eqref{eq:18} as follows:
 \begin{equation}
  \label{eq:33}
   \begin{split}
    &\quad\sup_{s\in I_\rho^{\theta_0}}\int_{B_\rho}
    (\tilde{u}(s)-k)_+^2\tilde{\eta}^2(s)\,dx
    +\iint_{Q_\rho^{\theta_0}}
    |\nabla(\tilde{u}-k)_+|^2\tilde{\eta}^2\,dsdx \\
    &\leq C(m)\Biggl\{
    \iint_{Q_\rho^{\theta_0}}(\tilde{u}-k)^2_+
    \biggl\{\Bigl(\frac{M}{\mu^+}\Bigr)^{1-\frac1m}\partial_s\tilde{\eta}^2
    +|\nabla\tilde{\eta}|^2\biggr\}
    \,dsdx \\
    &\quad+\tilde{h}(\rho,\omega)
    \biggl(\int_{I_\rho^{\theta_0}}
    m_n\bigl(B_\rho\cap\{\tilde{u}(s)>k\}\bigr)^{q'(\frac12-\frac1p)}
    \,ds\biggr)^{\frac2{q'}}
    \Biggr\} \\
    &\leq C(m)\Biggl\{
    \iint_{Q_\rho^{\theta_0}}(\tilde{u}-k)^2_+
    \biggl\{\partial_s\tilde{\eta}^2
    +|\nabla\tilde{\eta}|^2\biggr\}
    \,dsdx \\
    &\quad+\tilde{h}(\rho,\omega)
    \biggl(\int_{I_\rho^{\theta_0}}
    m_n\bigl(B_\rho\cap\{\tilde{u}(s)>k\}\bigr)^{q'(\frac12-\frac1p)}
    \,ds\biggr)^{\frac2{q'}}
    \Biggr\}
   \end{split} 
\end{equation}
 where $\tilde{h}(\rho,\omega):=
 \big\|\tilde{{f}}\big\|^2_{L^q(L^p_\mathrm{w})(Q_{\rho})}
 +\omega\|\tilde{g}\|_{L^\frac{q}2(L^{\frac{p}2}_\mathrm{w})(Q_{\rho})}$.

 We take $p_*,q_*>0$ as in the proof of Proposition \ref{prop:1} and for
 $i\in\N$ we take $\rho=\rho_i,\,k=k_i,\,\tilde{\eta}=\tilde{\eta_i}$
 satisfying $\tilde{\eta_i}\equiv1$ on $Q^{\theta_0}_{\rho_{i+1}}$ and
 \[
 \begin{aligned}
  k_i&=\mu^+-\frac\omega{2^{q_0}}+\frac{1}{2^{q_0+i+2}}\omega,\quad
  \rho_i=\frac12\rho+\frac1{2^{i+2}}\rho,  \\
  Y_i&:=\frac{m_{n+1}\bigl(Q^{\theta_0}_{\rho_i}\cap\{\tilde{u}>k_i\}\bigr)}%
  {m_{n+1}\bigl(Q^{\theta_0}_{\rho_0}\bigr)}, \\
  Z_i&=\frac{\rho_0^2}{m_{n+1}\bigl(Q^{\theta_0}_{\rho_0}\bigr)}
  \biggl(
  \int_{I_{\rho_i}^{\theta_0}}
  m_n\Bigl(B_{\rho_i}\cap\{\tilde{u}(s)>k_i\}\Bigr)^\frac{q_*}{p_*}\,ds
  \biggr)^\frac2{q_*}, \\
  |\nabla\tilde{\eta_i}|&\leq\frac{2}{\rho_i-\rho_{i+1}}
  \leq\frac{12\cdot2^i}{\rho_0}, \quad
  \partial_s\tilde{\eta_i}
  \leq\frac{4}{\theta_0}\frac{1}{\rho_i^2-\rho_{i+1}^2}
  \leq\frac{48\cdot2^{2i}}{\theta_0\rho^2}.
 \end{aligned}
 \]
 From \eqref{eq:33} and $(\tilde{u}-k_i)_+\leq\frac\omega{2^{q_0+1}}$, 
 we obtain
 \begin{equation*}
  \begin{split}
   &\quad\|(\tilde{u}-k_i)_+\tilde{\eta_i}\|_{L^\infty(L^2)\cap
   L^2(\dot{H}^1)(Q_{\rho_i}^{\theta_0})}^2 
   +\tilde{h}(\rho,\omega)
   \biggl(\int_{I_{\rho_i}^{\theta_0}}
   m_n\Bigl(B_\rho\cap\{\tilde{u}(s)>k_i\}\Bigr)^{q'(\frac12-\frac1p)}
   \,ds\biggr)^{\frac2{q'}}
   \\
   &\leq C(m)
   \Biggl\{
   \Bigl(\frac\omega{2^{q_0}}\Bigr)^2
   \Bigl(\frac1{\theta_0}+1\Bigr)
   \frac{2^{2i}}{\rho^2}m_{n+1}\Bigl(Q_{\rho_i}^{\theta_0}\cap\{\tilde{u}>k_i\}\Bigr) \\
   &\quad+\tilde{h}(\rho,\omega)
   \biggl(\int_{I_{\rho_i}^{\theta_0}}
   m_n\Bigl(B_\rho\cap\{\tilde{u}(s)>k_i\}
   \Bigr)^{\frac{q_*}{p_*}}\,ds\biggr)^{\frac{2}{q_*}(1+\frac{2\sigma_0}{n})}
   \Biggr\} \\
   &\leq C(m,\theta_0)
   \frac{m_{n+1}\bigl(Q_{\rho_0}^{\theta_0}\bigr)}{\rho_0^2}
   \Bigl(\frac\omega{2^{q_0}}\Bigr)^2 \\
   &\quad\times\Biggl\{
   2^{2i}Y_i+
   \tilde{h}(\rho,\omega)
   \Bigl(\frac{2^{q_0}}\omega\Bigr)^2
   \Bigl(\frac{m_{n+1}\bigl(Q_{\rho_0}^{\theta_0}\bigr)}{\rho_0^2}
   \Bigr)^\frac{2\sigma_0}{n}
   Z_i^{1+\frac{2\sigma_0}n}   
  \Biggr\}.
  \end{split} 
\end{equation*}
 Since $\delta_1\leq\theta_0^{-\frac1{q'}}2^{-q_0}$, we have
 \[
 \tilde{h}(\rho,\omega)
 \Bigl(\frac{2^{q_0}}\omega\Bigr)^2
 \Bigl(\frac{m_{n+1}\bigl(Q_{\rho_0}^{\theta_0}\bigr)}%
 {\rho_0^2}\Bigr)^\frac{2\sigma_0}{n}
 \leq C(n,p,q,\theta_0)
 \]
 and hence
 \begin{equation*}
  \|(\tilde{u}-k_i)_+\tilde{\eta_i}\|_{L^\infty(L^2)\cap
   L^2(\dot{H}^1)(Q_{\rho_i}^{\theta_0})}^2 
  \leq C(n,m,p,q,\theta_0)
   \Bigl(\frac\omega{2^{q_0}}\Bigr)^2
   \frac{m_{n+1}\bigl(Q_{\rho_0}^{\theta_0}\bigr)}{\rho_0}
   \Bigl\{
  2^{2i}Y_i+Z_i^{1+\frac{2\sigma_0}{n}}
  \Bigr\}.
 \end{equation*}
 By the Lady\v{z}enskaja inequality (cf. Proposition \ref{prop:4}) and
 the H\"older inequality, we have
 \begin{multline*}
  \|(\tilde{u}-k_i)_+\tilde{\eta_i}\|^2_{L^2(Q^{\theta_0}_{\rho_i})}
  \leq \|(\tilde{u}-k_i)_+\tilde{\eta_i}
  \|^2_{L^{2+\frac4n}(Q^{\theta_0}_{\rho_i})}
  \|\chi_{\{\tilde{u}>k_i\}}\|^2_{L^{n+2}(Q^{\theta_0}_{\rho_i})} \\
  \leq C(n,m,p,q,\theta_0)
  \Bigl(\frac\omega{2^{q_0}}\Bigr)^2
  \frac{m_{n+1}\bigl(Q_{\rho_0}^{\theta_0}
  \bigr)^{1+\frac{2}{n+2}}}{\rho_0^2}Y_i^{\frac2{n+2}}
  \Bigl\{
  2^{2i}Y_i+Z_i^{1+\frac{2\sigma_0}{n}}
  \Bigr\}
 \end{multline*}
 and
\begin{equation*}
    \|(\tilde{u}-k_i)_+\tilde{\eta_i}
  \|^2_{L^{q_*}(L^{p_*})(Q^{\theta_0}_{\rho_i})} 
  \leq C(n,m,p,q,\theta_0)
  \Bigl(\frac\omega{2^{q_0}}\Bigr)^2
  \frac{m_{n+1}\bigl(Q_{\rho_0}^{\theta_0}\bigr)}{\rho_0^2}
  \Bigl\{
  2^{2i}Y_i+Z_i^{1+\frac{2\sigma_0}{n}}
  \Bigr\}.
\end{equation*} 

 Since
 \[
 \begin{split}
  \|(\tilde{u}-k_i)_+\tilde{\eta_i}\|^2_{L^2(Q^{\theta_0}_{\rho_i})}
  &\geq\|(\tilde{u}-k_i)_+
  \|^2_{L^2(Q^{\theta_0}_{\rho_{i+1}}\cap\{\tilde{u}>k_{i+1}\})} \\
  &\geq (k_{i+1}-k_i)^2m_{n+1}
  \Bigl(Q^{\theta_0}_{\rho_{i+1}}\cap\{\tilde{u}>k_{i+1}\}\Bigr) \\
  &=\biggl(\frac{\omega}{2^{q_0+i+3}}\biggr)^2
  m_{n+1}\bigl(Q_{\rho_0}^{\theta_0}\bigr)Y_{i+1}
 \end{split} 
 \]
 and
 \[
 \begin{split}
  \|(\tilde{u}-k_i)_+\tilde{\eta_i}
  \|^2_{L^{q_*}(L^{p_*})(Q^{\theta_0}_{\rho_i})}
  &\geq\|(\tilde{u}-k_i)_+
  \|^2_{L^{q_*}L^{p_*}(Q^{\theta_0}_{\rho_{i+1}}\cap\{\tilde{u}>k_{i+1}\})} \\
  &\geq \biggl(\frac{\omega}{2^{q_0+i+3}}\biggr)^2
  \frac{m_{n+1}\bigl(Q_{\rho_0}^{\theta_0}\bigr)}{\rho_0^2}Z_{i+1},
 \end{split} 
 \]
 we obtain
 \[
 Y_{i+1}
 \leq C(n,m,p,q,\theta_0)
 \Bigl\{
 2^{4i}Y_i^{1+\frac2{n+2}}
 +2^{2i}Y_i^{\frac2{n+2}}Z_i^{1+\frac{2\sigma_0}{n}} \Bigr\} 
 \]
 and
 \[
 Z_{i+1}
 \leq C(n,m,p,q,\theta_0)
 \Bigl\{
 2^{4i}Y_i
 +2^{2i}Z_i^{1+\frac{2\sigma_0}{n}} \Bigr\}. 
 \]

Considering the same calculation of \eqref{eq:42} and \eqref{eq:43}, we
 obtain
 \[
  Z_0\leq
 \left\{
 \begin{aligned}
  &C(n,p,q,\theta_0)Y_0^\frac{2}{p_*}\quad\text{if}\quad q\geq p, \\
  &C(n,p,q,\theta_0)Y_0^\frac{2}{q_*}\quad\text{if}\quad q< p. \\
 \end{aligned}
 \right.
 \]
 Therefore, by Lemma \ref{lem:9},
 there exists
 $0<\nu=\nu(n,m,p,q,\theta_0)<1$ such that if
 $Y_0\leq \nu$, then
 $Y_i\rightarrow0$ as $i\rightarrow\infty$, i.e.
 \[
 \tilde{u}(s,x)<\mu^+-\frac\omega{2^{q_0+2}}
 \quad\text{a.a.}\ (s,x)\in Q^{\theta_0}_{\frac\rho2}.
 \]
 By Lemma \ref{lem:7}, we obtain the upper bounds of $u$.
\end{proof}


\appendix


\section{Appendix}
Their results are well-known, however we give the proof here for
reader's convenience.

\subsection{Some Sobolev type inequality}

\begin{proposition}
 [Lady{\v{z}}enskaja-Solonnikov-Ural'ceva {\cite[p.74]{MR0241822}}]
 \label{prop:4}
 Let $I\subset\R$ be an open interval and let $\Omega\subset\R^n$ be a
 domain. Then for $f\in L^\infty(I;L^2(\Omega))\cap L^2(I;H^1_0(\Omega))$
 and $p,q\geq2$ satisfying
\[
  \begin{aligned}
  &\frac2q+\frac{n}p=\frac{n}2&&\text{if}\quad n\neq2, \\
  &\frac2q+\frac{n}p=\frac{n}2\quad\text{without}\quad q=2,\ p=\infty
   &&\text{if}\quad  n=2,
 \end{aligned}
\]
we obtain
\begin{equation}
 \label{eq:34}
  \|f\|_{L^q(I;L^p(\Omega))}
  \leq C(n,p,q)
  (\|f\|_{L^\infty(I;L^2(\Omega))}+\|\nabla f\|_{L^2(I\times\Omega)}).
\end{equation}  
\end{proposition}

\begin{proof}%
 By the Gagliardo-Nirenberg-Sobolev inequality, we have
 \[
  \|f(t)\|_{L^p(\Omega)}
 \leq C(n,p)\|\nabla f(t)\|_{L^2(\Omega)}^{\frac{n}2-\frac{n}p}
 \|f(t)\|_{L^2(\Omega)}^{1-(\frac{n}2-\frac{n}p)}
 \quad\text{a.a.}\ t\in I.
 \]
 Taking $L^q(I)$ norm on both side, we obtain \eqref{eq:34}.
\end{proof}

\begin{proposition}%
  [Lady{\v{z}}enskaja-Solonnikov-Ural'ceva~{\cite[p.91]{MR0241822}}]
 \label{prop:3} Let $f$ be a non-negative function belonging to
 $W^{1,1}(B_\rho)$ and let $l>k$. Then there exists a constant $C>0$
 depending on $n$ only such that
  \[
   (l-k)m_n\bigl(\{f>l\}\bigr)
  \leq\frac{C\rho^{n+1}}{m_n\bigl(B_\rho\bigr)-m_n\bigl(\{f>k\}\bigr)}
  \int_{\{k<f\leq l\}}|\nabla f|\,dx.
  \]
 \end{proposition}
 For the proof of Proposition \ref{prop:3}, we need the following
 Poincar\'e inequalities:
 \begin{lemma}%
  [Lady{\v{z}}enskaja-Solonnikov-Ural'ceva~{\cite[Lemma 5.1 in p.89]{MR0241822}}]
  \label{lem:8}
  Let $g\in W^{1,1}(B_{\rho})$ ne a non-negative function 
  and let $N_0:=\{g=0\}$. Let $\eta(x)=\eta(|x|)$ be a decreasing
  function of $|x|$ satisfying $0\leq\eta\leq1$ and
  $\eta\big|_{N_0}\equiv1$. Then for measurable set $N\subset B_\rho$,
  we have
  \[
   \int_Ng(x)\eta(x)\,dx
  \leq \frac{C_n\rho^n}{m_n(N_0)}m_n(N)^\frac1n\int_{B_\rho}|\nabla g(x)|\eta(x)\,dx.
  \]
 \end{lemma}
 
 \begin{proof}%
  First we consider the case of $n\geq2$. For $x\in N\,,\,x'\in N_0$, we
  have
  \[
   g(x)=g(x)-g(x')=-\int_0^{|x'-x|}\frac{d}{dr}g(x+r\omega)\,dr
  \leq\int_0^{|x'-x|}|\nabla g(x+r\omega)|\,dr
  \]
  where $\omega=\frac{x'-x}{|x'-x|}$. We now show
  \begin{equation}
   \label{eq:23}
    \eta(x)\leq\eta(x+r\omega)\quad\text{for}\quad 0<r\leq|x'-x|.
  \end{equation}  
  Either if $|x|\leq|x'|$, then $x+r\omega\in B_{|x'|}$ by the convexity
  of $B_{|x'|}$. By the monotonicity of $\eta$, we have
  $\eta(x+r\omega)\geq\eta(x')=1$.
  Otherwise, if $|x|>|x'|$, then $x+r\omega\in B_{|x|}$. Since
  $\eta(x+r\omega)\geq\eta(x)$, we obtain \eqref{eq:23}. 
  
  By \eqref{eq:23}, we have
  \[
   g(x)\eta(x)
  \leq\int_0^{|x'-x|}|\nabla g(x+r\omega)|\eta(x+r\omega)\,dr.
  \]
  Integrating over $x\in N$ and $x'\in N_0$, we have
  \[
   m_n(N_0)\int_Ng(x)\eta(x)\,dx
  \leq\int_N\,dx\int_{N_0}\,dx'\int_0^{|x'-x|}
  |\nabla g(x+r\omega)|\eta(x+r\omega)\,dr.
  \]
  Let $g(x)=\eta(x)=0$ on $x\in \R^n\setminus
  B_\rho$. Introducing the polar coordinate, we obtain
  \[
   \begin{split}
    &\quad\int_{N_0}\,dx'\int_0^{|x'-x|}|\nabla
    g(x+r\omega)|\eta(x+r\omega)\,dr \\
    &\leq \int_{B_{2\rho(x)}}\,dx'\int_0^{|x'-x|}|\nabla
    g(x+r\omega)|\eta(x+r\omega)\,dr \\
    &\leq \int_{B_{2\rho(x)}}\,dx'\int_0^{|x'-x|}|\nabla
    g(x+r\omega)|\eta(x+r\omega)\,dr \\
    &=\int_0^{2\rho}s^{n-1}\,ds\int_{\Sp^{n-1}}\,d\sigma
    \int_0^s\frac{|\nabla
    g(x+r\omega)|\eta(x+r\omega)}{r^{n-1}}r^{n-1}\,dr 
    \quad (x'=s\sigma+x)
    \\ 
    &=\int_0^{2\rho}s^{n-1}\,ds\int_{B_s(x)}\frac{|\nabla
    g(y)|\eta(y)}{|x-y|^{n-1}}\,dy
    \quad (y=x+r\sigma)\\
    &\leq\frac{(2\rho)^n}{n}\int_{B_\rho}\frac{|\nabla
    g(y)|\eta(y)}{|x-y|^{n-1}}\,dy.
   \end{split}
  \]
  Therefore, 
  \[
  \begin{split}
   m_n(N_0)\int_Ng(x)\eta(x)\,dx
   &\leq\frac{(2\rho)^n}{n}\int_N\,dx\int_{B_\rho}\frac{|\nabla
   g(y)|\eta(y)}{|x-y|^{n-1}}\,dy \\
   &=\frac{(2\rho)^n}{n}\int_{B_\rho}|\nabla
   g(y)|\eta(y)\,dy\int_{N}\frac{1}{|x-y|^{n-1}}\,dx.
  \end{split}  
  \]
  We now show the following estimate:
  \begin{equation}
   \label{eq:38}
    \int_{N}\frac{1}{|x-y|^{n-1}}\,dx
    \leq (1+\mathscr{H}^{n-1}(\Sp^{n-1}))m_n(N)^\frac1n,
  \end{equation}  
  where $\mathscr{H}^{n-1}(\Sp^{n-1})$ is the $(n-1)$-dimensional
  Hausdorff measure of the $(n-1)$-dimensional unit sphere.
  To show \eqref{eq:38}, let $\delta>0$ to be chosen
  later. We split the integral
  \[
  \begin{split}
   &\quad\int_{N}\frac{1}{|x-y|^{n-1}}\,dx \\
   &\leq \int_{N\cap\{|x-y|\leq\delta\}}\frac{1}{|x-y|^{n-1}}\,dx
   +\int_{N\cap\{|x-y|\geq\delta\}}\frac{1}{|x-y|^{n-1}}\,dx 
   =:I_1+I_2.
  \end{split}  
  \]
  By a simple calculation, we obtain
  \[
   \begin{split}
    I_1&\leq\int_0^\delta\frac{r^{n-1}}{r^{n-1}}\,dr\int_{\Sp^{n-1}}\,d\sigma
    =\delta\mathscr{H}^{n-1}(\Sp^{n-1}), \\
    I_2&\leq\int_N\frac1{\delta^{n-1}}\,dx\leq\delta^{1-n}m_n(N).
   \end{split}  
  \]
  Taking $\delta=m_n(N)^\frac1n$, we have
  $I_1+I_2\leq(1+\mathscr{H}^{n-1}(\Sp^{n-1}))m_n(N)^\frac1n$ and we
  obtain \eqref{eq:38}.

  Using \eqref{eq:38}, we have
\begin{equation*}
    m_n(N_0)\int_{N}g(x)\eta(x)\,dx 
   \leq\frac{2^n(1+\mathscr{H}^{n-1}(\Sp^{n-1}))}{n}\rho^nm_n(N)^\frac1n\int_{B_\rho}|\nabla
   g(y)|\eta(y)\,dy.
\end{equation*}
  We consider the case $n=1$. For $x\in N$ and $x'\in N_0$, we
  have
  \[
   g(x)=g(x)-g(x')=\int_{x'}^x\frac{d}{dy}g(y)\,dy
  \leq\left|\int_{x'}^x\biggl|\frac{d}{dy}g(y)\biggr|\,dy\right|.
  \]
  Since 
  \[
   g(x)\eta(x)
  \leq\left| \int_{x'}^x\biggl|\frac{d}{dy}g(y)\eta(y)\biggr|\,dy\right|
  \leq\int_{1}^{-1}\biggl|\frac{d}{dy}g(y)\eta(y)\biggr|\,dy,
  \]
  we obtain
  \[
   \int_Ng(x)\eta(x)\,dx
  \leq m_n(N)\int_{-1}^1|\nabla g(x)|\eta(x)\,dx.
  \]
 \end{proof}

\begin{proof}%
 [Proof of Proposition \ref{prop:3}]
 Let 
 \[
  \begin{aligned}
   g(x)&:=\max\{l-k\,,\,(f-k)_+\}\in W^{1,1}(B_\rho), &
   N_0&:=\{f<k\},\\
   \eta(x)&\equiv1, & N&:=\{f>l\}.\\
  \end{aligned}
 \]
 Then, by the Lemma \ref{lem:8}, we have
 \[
 \int_Ng(x)\,dx
 \leq\frac{C_n\rho^nm_n\bigl(N\bigr)^\frac1n}%
 {m_n\bigl(N_0\bigr)}\int_{B_\rho}|\nabla g(x)|\,dx,
 \]
 hence
 \[
  (l-k)m_n\bigl(\{f>l\}\bigr)
 \leq\frac{C_n\rho^nm_n\bigl(\{f>l\}\bigr)^\frac1n}%
 {m_n\bigl(\{f<k\}\bigr)}\int_{\{k<f\leq l\}}|\nabla f(x)|\,dx.
 \]
\end{proof}

\subsection{The recursive inequalities}

\begin{lemma}%
 [Lady{\v{z}}enskaja-Solonnikov-Ural'ceva~{\cite[Lemma 5.7 in p.96]{MR0241822}}]
 \label{lem:9}
 Let $C,\varepsilon,\delta>0$ and $b\geq1$. Assume that sequences
 $\{Y_n\}_{n=0}^\infty\,,\,\{Z_n\}_{n=0}^\infty\subset(0,\infty)$
 satisfy
 \begin{equation}
  \label{eq:35}
  \begin{split}
   Y_{n+1}&\leq Cb^n(Y_n^{1+\delta}+Y_n^\delta Z_n^{1+\varepsilon}), \\
   Z_{n+1}&\leq Cb^n(Y_n+Z_n^{1+\varepsilon}).
  \end{split}   
 \end{equation}
 Let 
 \[
 d:=\min\left\{\delta,\frac{\varepsilon}{1+\varepsilon}\right\}\,,\,
 \lambda=\min\left\{(2C)^{-\frac1\delta}b^{-\frac1{\delta d}}\,,\,
 (2C)^{-\frac{1+\varepsilon}{\varepsilon}}b^{-\frac1{\varepsilon d}}\right\}.
 \]
 Then, if $Y_0\leq\lambda$ and $Z_0\leq\lambda^\frac{1}{1+\varepsilon}$,
 we obtain
 \begin{equation}
  \label{eq:36}
   Y_n\leq\lambda b^{-\frac{n}{d}}\,,\,
   Z_n\leq(\lambda b^{-\frac{n}{d}})^\frac{1}{1+\varepsilon}.
  \end{equation}
 \end{lemma}

\begin{proof}%
 Inequalities \eqref{eq:36} are valid for $n=0$. We prove \eqref{eq:36} by
 induction. If \eqref{eq:36} hold for $n$, then by \eqref{eq:35}, we have
 \[
 Y_{n+1}\leq 2C\lambda^{1+\delta}b^{n(1-\frac{1+\delta}{d})}\,,\,
 Z_{n+1}\leq 2C\lambda b^{n(1-\frac1d)}.
 \]
 Since $\lambda\leq(2C)^{-\frac1\delta}b^{-\frac1{\delta d}}$ and
 $d\leq\delta$, we have
 \[
 2C\lambda^{1+\delta}b^{n(1-\frac{1+\delta}{d})}
 \leq\lambda b^{-\frac1d}b^{-\frac{n}{d}+n(1-\frac{\delta}{d})}
 \leq \lambda b^{-\frac{n+1}{d}}.
 \]
 Similarly, since
 $\lambda\leq(2C)^{-\frac{1+\varepsilon}{\varepsilon}}b^{-\frac1{\varepsilon
 d}}$, we obtain
 \[
 \begin{split}
  2C\lambda b^{n(1-\frac1d)}
  &=2C\lambda^\frac{\varepsilon}{1+\varepsilon}
  \lambda^\frac{1}{1+\varepsilon}b^{-\frac{n+1}{(1+\varepsilon)d}}
  b^{n(1-\frac1d)+\frac{n+1}{(1+\varepsilon)d}} \\
  &\leq(\lambda b^{-\frac{n+1}{d}})^\frac{1}{1+\varepsilon}
  b^{n(1-\frac{\varepsilon}{(1+\varepsilon)d})}.
 \end{split} 
\]
 Since $d\leq\frac{\varepsilon}{1+\varepsilon}$, we find
 $1-\frac{\varepsilon}{(1+\varepsilon)d}\leq0$ and hence we have \eqref{eq:36}
 for $n+1$.
\end{proof}

\subsection{The weak $L^p$ spaces and the Lorentz spaces}
Let $\Omega\subset\R^n$ be a domain (not necessary bounded).

 \begin{definition}[The Lorentz spaces]
  For $1\leq p< \infty$, we define the Lorentz space
  $L^{p,\infty}(\Omega)$ by
  \[
  L^{p,\infty}(\Omega)
  :=\{f\in L^1_{\mathrm{loc}}(\Omega)\,:\,\lambda^p\mu_{|f|}(\lambda)\ \text{is
  bounded for all}\ \lambda>0\}
  \]
  where
  $\mu_{|f|}(\lambda):=m_n\bigl(\{|f|>\lambda\}\bigr)$. 
 \end{definition}

 \begin{proposition}
  [cf. Benilan-Brezis-Crandall {\cite[pp.548]{MR0390473}}]
  \label{prop:5}
  For $1<p<\infty$, we have
  \[
  \frac{p-1}{p^{1+\frac1p}}\|f\|_{L^p_\mathrm{w}(\Omega)}
  \leq \sup_{\lambda>0}\lambda\mu_{|f|}(\lambda)^\frac1p
  \leq \|f\|_{L^p_\mathrm{w}(\Omega)}.
 \]
 \end{proposition}

\begin{proof}%
 First, we show $\sup_{\lambda>0}\lambda\mu_{|f|}(\lambda)^\frac1p
 \leq \|f\|_{L^p_{\mathrm{w}}(\Omega)}$. For $\rho,\lambda>0$, 
 we take $K=\{|f|>\lambda\}\cap B_\rho$. 
 Then we have
 \[
 \begin{split}
  \|f\|_{L^p_{\mathrm{w}}(\Omega)}
  &\geq
  m_n\bigl(\{x\in \Omega\cap B_\rho:|f(x)|>\lambda\}\bigr)^{\frac1p-1}
  \int_{\{|f|>\lambda\}\cap B_\rho}
  |f(x)|\,dx \\
  &\geq\lambda m_n\bigl(\{x\in \Omega\cap B_\rho:|f(x)|>\lambda\}\bigr)^\frac1p.
 \end{split}
\]
 Letting $\rho\rightarrow\infty$, we find
 \[
  \lambda\mu_{|f|}(\lambda)^\frac1p\leq\|f\|_{L^p_{\mathrm{w}}(\Omega)}.
 \]
 Second, we show
 $\frac{p-1}{p^{1+\frac1p}}\|f\|_{L^p_{\mathrm{w}}(\Omega)} \leq
 \sup_{\lambda>0}\lambda\mu_{|f|}(\lambda)^\frac1p$.  We fix
 $\lambda_0>0$. For measurable set $K\subset\Omega$, we have
 \[
  \int_K|f(x)|\,dx
 \leq \lambda_0m_n\bigl(K\bigr)+\int_{\{|f|>\lambda_0\}}|f(x)|\,dx.
 \]
 By the above inequality, we have
 \[
 \begin{split}
  \int_{\{|f|>\lambda_0\}}|f(x)|\,dx
  &=\int_0^\infty 
  m_n\Big(\{x\in\{|f|>\lambda_0\}:|f(x)|>\lambda\}\Big)\,d\lambda \\
  &=\int_0^{\lambda_0}m_n\bigl(\{|f|>\lambda_0\}\bigr)\,d\lambda
  +\int_{\lambda_0}^\infty m_n\bigl(\{|f|>\lambda\}\bigr)\,d\lambda \\
  &=\lambda_0\mu_{|f|}(\lambda_0)
  +\int_{\lambda_0}^\infty\mu_{|f|}(\lambda)\,d\lambda \\
  &\leq\lambda_0^{1-p}\sup_{\lambda>0}\lambda^p\mu_{|f|}(\lambda)
  +\sup_{\lambda>0}\lambda^p\mu_{|f|}(\lambda)
  \int_{\lambda_0}^\infty\lambda^{-p}\,d\lambda \\
  &=\frac{p}{p-1}\lambda_0^{1-p}
  \sup_{\lambda>0}\lambda^p\mu_{|f|}(\lambda).
 \end{split} 
 \]
 Taking
 $\lambda_0^pm_n\bigl(K\bigr)=p\sup_{\lambda>0}\lambda^p\mu_{|f|}(\lambda)$,
 we find
 \[
 \int_K|f(x)|\,dx
 \leq \frac{p}{p-1}(p\sup_{\lambda>0}\lambda^p\mu_{|f|}(\lambda))^\frac1p|K|^{1-\frac1p}
 \]
 or
 \[
 \|f\|_{L^p_\mathrm{w}(\Omega)} 
 \leq \frac{p^{1+\frac1p}}{p-1}\sup_{\lambda>0}\lambda\mu_{|f|}(\lambda)^\frac1p.
 \]
\end{proof}


\textbf{Acknowledgments}: The author would like to express his deepest
gratitude to Professor Takayoshi Ogawa for his valuable comments and
encouragement. The author also wishes to thank the referee for his
valuable suggestions. This work is supported by the JSPS Research
Fellowships for Young Scientists and the JSPS Grant-in-Aid for the JSPS
fellows \#21-1281.




\nocite{MR2286292,MR1218742}

\end{document}